\documentclass[11pt, a4paper]{article}

\usepackage{latexsym}
\usepackage{amsmath, epsfig}
\usepackage{color}
\usepackage{graphicx, float}
\usepackage{amssymb,amsfonts,amsthm,graphicx,url}
\usepackage[T1]{fontenc}
\usepackage[utf8]{inputenc}

\definecolor{VeryLightBlue}{rgb}{0.9,0.9,1}
\definecolor{LightBlue}{rgb}{0.8,0.8,1}
\definecolor{MidBlue}{rgb}{0.5,0.5,1}
\definecolor{DarkBlue}{rgb}{0,0,0.6}
\definecolor{Blue}{rgb}{0,0,1}
\definecolor{Gold}{rgb}{1,0.843,0}
\definecolor{LightGreen}{rgb}{0.88,1,0.88}
\definecolor{MidGreen}{rgb}{0.6,1,0.6}
\definecolor{DarkGreen}{rgb}{0,0.6,0}
\definecolor{VeryLightYellow}{rgb}{1,1,0.9}
\definecolor{LightYellow}{rgb}{1,1,0.6}
\definecolor{MidYellow}{rgb}{1,1,0.5}
\definecolor{DarkYellow}{rgb}{1,1,0.2}
\definecolor{DarkPurple}{rgb}{.6,0,1}
\definecolor{Red}{rgb}{1,0,0}
\definecolor{VeryLightRed}{rgb}{1,0.9,0.9}
\definecolor{LightRed}{rgb}{1,0.8,0.8}
\definecolor{MidRed}{rgb}{1,0.55,0.55}

\long\def\delete#1{}

\usepackage[normalem]{ulem}

\newtheorem{theorem}{Theorem}[section]
\newtheorem{lemma}[theorem]{Lemma}

\theoremstyle{definition}

\newtheorem{example}{Example}[section]

\input amssym.def
\input amssym.tex

\newcommand{\be}{\begin{equation}}
\newcommand{\ee}{\end{equation}}
\newcommand{\bea}{\begin{eqnarray}}
\newcommand{\eea}{\end{eqnarray}}
\newcommand{\bean}{\begin{eqnarray*}}
\newcommand{\eean}{\end{eqnarray*}}

\def\qed{\hfill$\Box$\vspace{11pt}}

\def\la{\langle}
\def\ra{\rangle}

\def\ZZZ{\mathbb{Z}}

\def\PP{{\cal P}}

\def\b0{{\bf 0}}

\def\Ga{\Gamma}

\def\Si{\Sigma}

\def\b{\beta}

\def\l{\lambda}

\def\z{\zeta}

\def\Cay{{\rm Cay}}

\def\SL{{\rm SL}}

\def\Cay{{\rm Cay}}

\begin{document}
\openup 0.3\jot

\title{Perfect codes in Cayley graphs of abelian groups}

\delete
{
\author[1]{Peter J. Cameron \thanks{E-mail: \texttt{pjc20@st-andrews.ac.uk}}}
\author[2]{Roro Sihui Yap\thanks{E-mail: \texttt{roro\_@outlook.sg}}}
\author[2]{Sanming Zhou\thanks{E-mail: \texttt{sanming@unimelb.edu.au}}}
\affil[1]{\small School of Mathematics and Statistics, University of St Andrews, St Andrews, Fife KY16 9SS, UK}
\affil[2]{\small School of Mathematics and Statistics, The University of Melbourne, Parkville, VIC 3010, Australia}
}

\author{\renewcommand{\thefootnote}{\arabic{footnote}}Peter J. Cameron\footnotemark[1] , Roro Sihui Yap\footnotemark[2] , Sanming Zhou\footnotemark[2]}

\footnotetext[1]{School of Mathematics and Statistics, University of St Andrews, St Andrews, Fife KY16 9SS, UK}

\footnotetext[2]{School of Mathematics and Statistics, The University of Melbourne, Parkville, VIC 3010, Australia}

\renewcommand{\thefootnote}{}
\footnotetext{{\em E--mail addresses}: \texttt{pjc20@st-andrews.ac.uk} (Peter J. Cameron), \texttt{roro\_@outlook.sg} (Roro Sihui Yap), \texttt{sanming@unimelb.edu.au} (Sanming Zhou)}

\date{}

\maketitle

\begin{abstract}
A perfect code in a graph $\Ga = (V, E)$ is a subset $C$ of $V$ such that no two vertices in $C$ are adjacent and every vertex in $V \setminus C$ is adjacent to exactly one vertex in $C$. A total perfect code in $\Ga$ is a subset $C$ of $V$ such that every vertex of $\Ga$ is adjacent to exactly one vertex in $C$. In this paper we prove several results on perfect codes and total perfect codes in Cayley graphs of finite abelian groups. 

{\em Key words}: perfect code; total perfect code; efficient dominating set; tiling of finite groups; Cayley graph

{\em AMS Subject Classification (2010)}: 05C25, 05C69, 94B25 
\end{abstract}


\section{Introduction}
\label{sec:int}

Perfect codes have long been important objects of study in information theory \cite{Heden, vanLint}. Examples of perfect codes include the well-known Hamming and Golay codes. The notion of perfect codes can be generalized to graphs \cite{Biggs73, Kra} in such a way that $q$-ary perfect $t$-codes of length $n$ under the Hamming metric are precisely perfect $t$-codes in the Hamming graph $H(n, q)$ and those under the Lee metric are exactly perfect $t$-codes in the Cartesian product $L(n, q)$ of $n$ cycles of length $q$. Since both $H(n, q)$ and $L(n, q)$ are Cayley graphs, perfect codes in Cayley graphs can be viewed as a generalization of perfect codes in the classical setting. As such there has been considerable interest in studying perfect codes in Cayley graphs; see, for example, \cite{DS03, DSLW16, E87, FHZ17, HXZ, L01, MBG07, T04, Z15, Z16}. In this paper we focus on  perfect codes in Cayley graphs of abelian groups. 
 
All graphs considered in this paper are finite, undirected and simple, and all groups considered are finite. 
Let $\Ga = (V, E)$ be a graph and $t \ge 1$ an integer. The \emph{ball} with radius $t$ and centre $u \in V$ (that is,  the closed $t$-neighbourhood of $u$ in $\Ga$) is the set of vertices of $\Ga$ with distance at most $t$ to $u$ in $\Ga$. A subset $C$ of $V$ is called \cite{Biggs73, Kra} a \emph{perfect $t$-code} in $\Ga$ if the balls with radius $t$ and centres in $C$ form a partition of $V$. It is readily seen that $q$-ary perfect $t$-codes of length $n$ under the Hamming or Lee metric are precisely perfect $t$-codes in $H(n, q)$ or $L(n, q)$, respectively. In graph theory, perfect $1$-codes in graphs have long been studied under the name of efficient dominating sets \cite{DS03} and independent perfect dominating sets \cite{L01}. Since we only consider perfect $1$-codes in this paper, from now on we simply call them \emph{perfect codes}. A subset $C$ of $V$ is called \cite{AD14, CHKS, D08, GHT, KPY} a \emph{total perfect code} in $\Ga$ if every vertex of $\Ga$ has exactly one neighbour in $C$. (Thus, the subgraph of $\Ga$ induced by a total perfect code $C$ is a matching and so $|C|$ must be even.) As noticed in \cite{Z16}, total perfect codes in $L(n, q)$ and in the infinite $n$-dimensional grid graph (which gives rise to the Manhattan distance) are exactly diameter perfect codes \cite{AAK01, Etzion11} with minimum distance $4$. Also, perfect codes and total perfect codes in regular graphs are essentially two special kinds of equitable partitions of the graph into two parts (see, for example, \cite{WXZ22}).  

Given a group $G$ with identity element $e$ and an inverse-closed subset $S$ of $G \setminus \{e\}$, the \emph{Cayley graph} $\Cay(G, S)$ of $G$ with respect to the \emph{connection set} $S$ is the graph with vertex set $G$ such that $x, y \in G$ are adjacent if and only if $yx^{-1} \in S$. In particular, Cayley graphs of cyclic groups are called \emph{circulant graphs}. Clearly, $\Cay(G, S)$ is connected if and only if $S$ is a generating set of $G$. As mentioned above, perfect codes in Cayley graphs can be considered as a generalization of perfect codes in coding theory. They are also closely related to factorizations and tilings of groups, where a \emph{factorization} \cite{SS} of a group $G$ into $n$ factors (where $n \ge 2$) is an $n$-tuple of subsets $(A_1, \ldots, A_n)$ of $G$ such that each element of $G$ can be uniquely expressed as $a_1 \cdots a_n$ with $a_i \in A_i$ for $1 \le i \le n$, and a \emph{tiling} \cite{Dinitz06} of $G$ is a factorization $(A, B)$ of $G$ with $e \in A \cap B$. The study of factorizations and tilings of groups has a long history (see, for example, \cite{Dinitz06,RT1966,Sands62,Szab2006,SS}). In particular, G. Haj\'{o}s \cite{Hajos1942} proved that, for any factorization $(A_1, \ldots, A_n)$ of an abelian group $G$ with each $A_i$ an arithmetic progression, at least one $A_j$ must be a coset of a subgroup of $G$. It can be easily verified that $(A, B)$ is a tiling of a group $G$ such that $A$ is inverse-closed if and only if $B$ is a perfect code in $\Cay(G, A \setminus \{e\})$ with $e \in B$. Thus, the study of perfect codes in Cayley graphs is essentially the study of tilings of groups. From the perspective of perfect codes, the basic problems in this area of study are the existence, classification and characterization of perfect codes in a given Cayley graph, or equivalently that of the ``complement'' $B$ of a given inverse-closed subset $A$ containing $e$ such that $(A, B)$ is a tiling of $G$. 

Perfect codes in Cayley graphs of abelian groups deserve special attention due to the following reasons: (i) in the study of factorizations and tilings, abelian groups are the original and most important case; (ii) since any finite abelian group is of the form $G = \mathbb{Z}_{q_1} \times \cdots \times \mathbb{Z}_{q_n}$, the elements of a perfect code in a Cayley graph of $G$ are codewords in the usual sense with $i$-th bit from the alphabet $\mathbb{Z}_{q_i}$, and such a perfect code is a mixed perfect code \cite{Heden} if not all $q_1, \ldots, q_n$ are equal. Perfect codes in circulant graphs have received considerable attention in recent years \cite{YPD14, KM13, OPR07, Z15}, but no classification of such perfect codes has been achieved even in this innocent-looking case. In this paper we study perfect codes in Cayley graphs of abelian groups. Our main results are as follows. In Section \ref{sec:good}, we prove that perfect codes in Cayley graphs of good abelian groups with degree $p-1$ are cosets of a subgroup (Theorem \ref{thm:APCsubgp}), where $p$ is a prime. We also obtain the counterpart of this result for total perfect codes (Theorem \ref{thm:ATPCsubgp}). In Section \ref{sec:abelian}, we give a necessary and sufficient condition for the existence of perfect codes in Cayley graphs of abelian groups with degree $p^i - 1$ (Theorems \ref{thm:APCforP} and \ref{thm:APCforPl}), where $i$ is $1$ or the exponent of $p$ in the order of the group involved. In Section \ref{sec:abelian}, we also obtain a similar necessary and sufficient condition for total perfect codes (Theorems \ref{thm:ATPCforP} and \ref{thm:ATPCforPl}). These four results in Section \ref{sec:abelian} generalize the main results in \cite{FHZ17} from cyclic groups to general abelian groups. In addition, in Section \ref{sec:abelian} we give a lemma (Lemma \ref{lem:redn}) which reduces the existence of perfect codes in a Cayley graph $\Cay(G, S)$ of an abelian group $G$ to the case when $S \cup \{0\}$ is aperiodic ( where $0$ is the identity element of $G$), and use this reduction  to prove a few known results. In Section \ref{sec:tpc}, we prove several results on total perfect codes in circulant graphs (see Table \ref{tab2} for a summary) which are parallel to some existing results on perfect codes in circulant graphs.


\section{Preliminaries}
\label{sec:lem}

The following lemma gives equivalent definitions of perfect codes and total perfect codes in Cayley graphs in terms of factorizations and tilings (see, for example, \cite{YPD14, HXZ}).   

\begin{lemma}
\label{lem:DirP}
Let $G$ be a group, $S$ an inverse-closed subset of $G \setminus \{e\}$, and $C$ a subset of $G$.  
\begin{itemize} 
\item[\rm (a)] $C$ is a perfect code in $\Cay(G,S)$ if and only if $(S \cup \{e\}, C)$ is a factorization of $G$ (see \cite{YPD14});
\item[\rm (b)] $C$ is a total perfect code in $\Cay(G,S)$ if and only if $(S, C)$ is a factorization of $G$.
\end{itemize}
\end{lemma}

Thus, if $C$ is a perfect code (respectively, total perfect code) in $\Cay(G, S)$, then $(|S|+1) |C| = |G|$ (respectively, $|S| |C| = |G|$).  

\begin{lemma}
[{\cite[Corollary 2.3]{HXZ}}]
\label{lem:subgpPC}
Let $G$ be a group and $H$ a normal subgroup of $G$.  
\begin{itemize}
\item[\rm (a)] If either $|H|$ or $|G|/|H|$ is odd, then there exists a Cayley graph of $G$ which admits $H$ as a perfect code.
\item[\rm (b)] If $|H|$ is even and $|G|/|H|$ is odd, then there exists a Cayley graph of $G$ which admits $H$ as a total perfect code.
\end{itemize}
\end{lemma}

The family of perfect codes in a graph is invariant under the automorphisms of the graph. Since $\Cay(G, S)$ admits the right regular representation of $G$ as a group of automorphisms, we can see that a subset $C$ of $G$ is a perfect code in $\Cay(G, S)$ if and only if $Cg$ is a perfect code in $\Cay(G, S)$ for every $g \in G$ (see, for example, \cite{L01}). Similarly, $C$ is a total perfect code in $\Cay(G,S)$ if and only if $Cg$ is a total perfect code in $\Cay(G,S)$ for every $g \in G$ (see \cite[Lemma 3.1]{Z16}). In particular, if $C$ is a (total) perfect code in $\Cay(G, S)$, then for any $g \in C$, $Cg^{-1}$ is a (total) perfect code in $\Cay(G, S)$ containing $e$. Therefore, any (total) perfect code in $\Cay(G, S)$ can be obtained from some (total) perfect code in $\Cay(G, S)$ containing $e$ by right multiplication by some element of $G$.

From now on we consider abelian groups only and use additive notation for their operation. We use $0$ to denote the identity element of an abelian group $G$. In view of the discussion above, $C$ is a (total) perfect code in $\Cay(G,S)$ if and only if the ``translation'' $C+g$ is also a (total) perfect code in $\Cay(G, S)$ for every $g \in G$, and any (total) perfect code in $\Cay(G,S)$ is a translation of some (total) perfect code in $\Cay(G, S)$ containing $0$. For subsets $X, Y$ and a subgroup $H$ of $G$, set $X+Y = \{x+y: x \in X, y \in Y\}$, $X_0 = X \cup \{0\}$, $X/H = \{H+x: x \in X\} = \{H+g \in G/H: (H+g) \cap X \ne \emptyset\}$ and $G_X = \{g \in G: X + g = X\}$. We also write $X-Y = \{x - y: x \in X, y \in Y\}$ and in particular $X - X = \{x-y: x, y \in X\}$. Write $G = X \oplus Y$ if $(X, Y)$ is a factorization of $G$. Since $X$ is invariant under $G_X$, $X$ is the union of some cosets of $G_X$ in $G$. In particular, if $0 \in X$, then $G_X \subseteq X$. In \cite[Section 2.2]{SS}, each element of $G_X$ is called a \emph{period} of $X$ and $G_X$ is called the \emph{subgroup of periods} of $X$. (Indeed, $G_X$ is a subgroup of $G$ as it is the setwise stabilizer of $X$ under the action of $G$ on itself by addition.) A nonempty subset $X$ of $G$ is called \emph{periodic} if $G_X \ne \{0\}$ and \emph{aperiodic} otherwise \cite{SS}. 

\begin{lemma} 
[{\cite[Proposition 2.1]{Vuza}}]
\label{lem:DirPprop}
Let $G$ be an abelian group, and let $X$ and $Y$ be nonempty subsets of $G$. Then $G = X \oplus Y$ if and only if any two of the following conditions are satisfied:  
\begin{itemize} 
\item[\rm (a)] $G = X + Y$;  
\item[\rm (b)] $(X - X) \cap (Y - Y) = \{0\}$;  
\item[\rm (c)] $|G| = |X| |Y|$.
\end{itemize}
\end{lemma}

Using Lemma \ref{lem:DirPprop}, the following lemma was proved in \cite{YPD14} under the condition that $0 \in X$. This result and its proof in \cite{YPD14} remain valid without this condition.  

\begin{lemma} 
[{\cite[Lemma 2.4]{YPD14}}]
\label{lem:perodiv}
Let $G$ be an abelian group. Let $X$ and $C$ be nonempty subsets of $G$ and let $H = G_{X}$. Then $G = X \oplus C$ if and only if $G/H = (X/H) \oplus (C/H)$ and $H \cap (C - C) = \{0\}$.
\end{lemma} 

\begin{lemma}
[\cite{{Vuza}}]
\label{lem:peripk} 
Let $n \ge 2$ be an integer, and let $X$ and $Y$ be nonempty subsets of $\ZZZ_n$ such that $|X| = p^k$ for a prime $p$ and an integer $k \geq 0$. If $\mathbb{Z}_n = X \oplus Y$, then at least one of $X$ and $Y$ is periodic.  
\end{lemma}

\begin{lemma}
\label{lem:shortpf}
Let $n \ge 2$ be an integer. Let $X$ be a subset of $G = \mathbb{Z}_n$ with $|X| \ge 2$ and let $H = G_{X}$. If $|X|$ and $n/|X|$ are coprime and $|X| = |H|$, then $x \not \equiv x'$ (mod $|X|$) for any two distinct elements $x, x' \in X$.
\end{lemma}

\begin{proof}
Set $k = |X|$. Since $|X| = |H|$ and $H$ is a subgroup of $G$, $k$ divides $n$ and $H = (n/k) \ZZZ_n$. Since $H = G_{X}$, $X$ is the union of some cosets of $H$. This together with $|X| = |H|$ implies that $X$ is a coset of $H$, say, $X = H+a$ for some $a \in X$. Thus, for any two distinct elements $x, x' \in X$, we have $x = (cn/k)+a$ and $x' = (c'n/k)+a$ for some distinct integers $c, c'$ with $0 \leq c, c' < k$. Since $k$ is not a divisor of $c - c'$ and $k$ and $n/k$ are coprime, it follows that $x - x' =  (c - c')n/k \not \equiv 0$ (mod $k$). 
\qed
\end{proof}

\begin{lemma}
\label{lem:rmvperi}
Let $G$ be an abelian group. Let $X$ be a nonempty subset of $G$ and let $H = G_X$. Then $X/H$ is an aperiodic subset of $G/H$.
\end{lemma}
	
\begin{proof}
Suppose that $H+g \in G/H$ satisfies $(X/H)+(H+g) = X/H$. Then for any $x \in X$ we have $H+(x + g) = (H+x) + (H+g) \in X/H$. So there exist $x' \in X$ and $h \in H$ such that $x + g = x' + h$. Since $H$ fixes $X$ setwise, we have $x + g = x' + h \in X$. Since this holds for every $x \in X$, we conclude that $X + g \subseteq X$. Since $X+g$ and $X$ have the same cardinality, it follows that $X+g = X$ and hence $g \in H$. Therefore, $H+g = H$. So the subgroup of periods of $X/H$ in $G/H$ is the trivial subgroup $\{H\}$ of $G/H$. In other words, $X/H$ is an aperiodic subset of $G/H$. 
\qed
\end{proof}


\section{Cayley graphs of good abelian groups}
\label{sec:good}

An abelian group is \emph{good} if for every factorization of it into two factors at least one factor is periodic \cite{SS}. In 1962, Sands \cite{Sands62} proved that good abelian groups are precisely the abelian groups of the following types and their subgroups: \{$p^k$, $q$\}, \{$p^2$, $q^2$\}, \{$p^2$, $q$, $r$\}, \{$p$, $q$, $r$, $s$\}, \{$p^3$, $2$, $2$\}, \{$p^2$, $2$, $2$, $2$\}, \{$p$, $2^2$, $2$\}, \{$p$, $2$, $2$, $2$, $2$\}, \{$p$, $q$, $2$, $2$\}, \{$p$, $3$, $3$\}, \{$3^2$, $3$\}, \{$2^\lambda$, $2$\}, \{$2^2$, $2^2$\}, \{$p$, $p$\}, where $p$, $q$, $r$ and $s$ are distinct primes. It is known (see, for example, \cite{Tij95}) that good cyclic groups are precisely $\mathbb{Z}_{p^k q}, \mathbb{Z}_{p^2q^2}, \mathbb{Z}_{p^2qr}, \mathbb{Z}_{pqrs}$ and their subgroups, where $p$, $q$, $r$ and $s$ are distinct primes and $k$ is a positive integer. 

\begin{theorem}
\label{thm:APCsubgp}
Let $G$ be a good abelian group with order $n$, and let $p$ be a proper odd prime divisor of $n$. The following statements hold:
\begin{itemize} 
\item[\rm (a)] 
every coset of any subgroup of $G$ with order $n/p$ is a perfect code in some Cayley graph of $G$ with degree $p-1$; 
\item[\rm (b)] 
every perfect code of size $n/p$ in any connected Cayley graph of $G$ with degree $p-1$ is a coset of some subgroup of $G$ with order $n/p$.  
\end{itemize}
\end{theorem}
	
\begin{proof}
(a) Let $H$ be a subgroup of $G$ with order $|H| = n/p$. Then $|G/H| = p$ as $|G|=n$. Since $p$ is odd, by Lemma \ref{lem:subgpPC}, $H$ is a perfect code in some Cayley graph $\Cay(G,S)$ of $G$. Thus, $|S_0| |H| = n$ and $\Cay(G,S)$ has degree $|S| = |S_0|-1 = p - 1$. Moreover, every coset of $H$ in $G$ is also a perfect code in $\Cay(G,S)$.  

(b) Now suppose that $C$ is a perfect code in some connected Cayley graph $\Cay(G,S)$ with degree $|S| = p - 1$. Then $|S_0| = |G|/|C| = p$ and $G = S_0 \oplus C$ by Lemma \ref{lem:DirP}. Since for any $g \in G$, $C+g$ is also a perfect code in $\Cay(G,S)$, we may assume without loss of generality that $0 \in C$. It suffices to prove that $C$ is a subgroup of $G$. Since $G$ is a good abelian group, either $S_0$ or $C$ is periodic.

\medskip	
\textsf{Claim 1.} $C$ is periodic.

Suppose otherwise. Then $S_0$ must be periodic; that is, $G_{S_0}$ is a nontrivial subgroup of $G$. Since $|G_{S_0}|$ divides $|S_0| = p$, we have $|G_{S_0}| = |S_0| = p$. Since $0 \in S_0$ and $S_0$ is the union of some cosets of $G_{S_0}$ in $G$, it follows that $S_0 = G_{S_0}$ and hence $S_0$ is a proper subgroup of $G$. Thus, $\Cay(G, S)$ is disconnected, but this contradicts our assumption. This proves Claim 1. 

Set $H = G_{C}$. Then by Claim 1, $H$ is a nontrivial subgroup of $G$, and $C$ is the union of some cosets of $H$. Moreover, $H \subseteq C$ as $0 \in C$. 

\medskip	
\textsf{Claim 2.} $S_0/ H$ is a subgroup of $G / H$.

In fact, by Lemma \ref{lem:perodiv}, $G / H = (C / H) \oplus (S_0 / H)$ and $H \cap (S_0 - S_0) = \{0\}$. Hence $H+s \neq H+s'$ for distinct $s, s' \in S_0$. Thus, $|S_0 / H| = |S_0| = p$. Since $G$ is a finite abelian group, $G / H$ is isomorphic to a subgroup of $G$. Since $G$ is a good abelian group, we obtain that $G/H$ is also a good abelian group. Since $C / H$ is an aperiodic subset of $G/H$ (Lemma \ref{lem:rmvperi}) and $G / H = (C / H) \oplus (S_0 / H)$, it follows that $S_0 / H$ is periodic. That is, the stabilizer $K$ of $S_0 / H$ in $G/H$ is nontrivial. Since $|K|$ is a divisor of $|S_0 / H| = p$, we must have $|K| = p$. Since $S_0 / H$ contains the identity element of $G/H$, we have $K \subseteq S_0 / H$ and therefore $S_0 / H = K$ is a subgroup of $G/H$, as stated in Claim 2. 
	
Since $\Cay(G,S)$ is connected, we have $G = \langle S_0 \rangle$. This together with Claim 2 implies that $G/H = \langle S_0 / H \rangle = S_0 / H$. However, by Lemma \ref{lem:DirPprop}, $|G/H| = |S_0 / H| |C/H|$. Thus, $|C| = |H|$, which together with $H \subseteq C$ implies $C = H$. 
\qed
\end{proof}

The following example shows that a similar statement as in part (b) of Theorem \ref{thm:APCsubgp} for total perfect codes is not true. More precisely, for a good abelian group $G$ with order $n$ and a proper odd prime divisor $p$ of $n$, a total perfect code in a connected Cayley graph of $G$ with degree $p$ may not be a coset of some subgroup of $G$ with order $n/p$. 

\begin{example}
Let $p \geq 5$ be a prime. Then $G = \mathbb{Z}_2 \times \mathbb{Z}_{2p}$ is a good abelian group. Let 
$$
S = \{(1, p)\} \ \cup \ \{(0, i): i \in \mathbb{Z}_{2p}, i \neq p \text{ is odd}\}
$$ 
and 
$$
C = \{(0,0), (0,1), (1,0), (1,1)\}.
$$ 
Then $\Cay(G,S)$ is connected as $G = \langle S \rangle$. Obviously, $|S| = p$, $|C| = 4$ and $|S| |C| = 4p = |G|$. We have $C - C = \{(0,0), (0, 1), (0, 2p-1), (1,0), (1,1), (1,2p-1)\}$. On the other hand, the second coordinate of each element of $S - S$ is even. Hence $(S - S) \cap (C - C) = \{(0,0)\}$. By Lemma \ref{lem:DirPprop}, $G = S \oplus C$, and by Lemma \ref{lem:DirP}, $C$ is a total perfect code in $\Cay(G,S)$. It can be verified that the only subgroup of $G$ with order $4$ is $H = \{(0,0), (0,p), (1,0), (1,p)\}$.
\delete
{
In fact, let $H$ be a subgroup of $G$ with order $4$. If $H$ contains an element $(x,y)$ such that $y \neq p$ is odd, then $y$ is coprime to $2p$ and therefore $|H| \geq |\langle (x,y) \rangle| \geq 2p > 4$, a contradiction. If $H$ contains an element $(x,y)$ such that $y \neq 0$ is even, then the greatest common divisor of $y$ and $2p$ is $2$ and therefore $|H| \geq |\langle (x,y) \rangle| \geq p > 4$, a contradiction again. 
}
Clearly, $C \ne H+g$ for any $g \in G$, and therefore $C$ is not a coset of any subgroup of $G$ with order $4$. 
\end{example}

The appropriate counterpart of part (b) of Theorem \ref{thm:APCsubgp} for total perfect codes is as follows. 

\begin{theorem}
\label{thm:ATPCsubgp}
Let $G$ be a good abelian group with order $n$, and let $p$ be a proper odd prime divisor of $n$. Let $C$ be a subset of $G$ with size $|C| = n/p$, and let $H = G_C$ be the subgroup of periods of $C$. If $C$ is a total perfect code in some connected Cayley graph $\Cay(G,S)$ of $G$, then $n$ is even and either (i) $|C| = 2$ and $\Cay(G,S) \cong K_{p, p}$, or (ii) $|C| > 2$, $|G/H| = 2p < n$, $H \cap S = \emptyset$, and $C$ is the union of exactly two cosets of $H$ in $G$. If in addition $0 \in C$, then $C = H \cup (H + z)$ for some $z \in G \setminus H$ such that $|(H + z) \cap S| = 1$ and $H + z$ generates the subgroup of periods of $S/H$ in $G/H$.
\end{theorem} 

\begin{proof}	
Suppose that $C$ is a total perfect code in a connected Cayley graph $\Cay(G,S)$ of $G$. Then $|C| = n/p$ is even and so $n$ is even. By Lemma \ref{lem:DirP}, $G = S \oplus C$, and by Lemma \ref{lem:DirPprop}, $|S| = |G|/|C| = p$. Since $G$ is a good abelian group, either $S$ or $C$ is periodic. 

If $S$ is periodic, then $|G_S| > 1$. Since $|G_S|$ divides $|S|$, we have $|G_S| = |S| = p$ and hence $S$ is a coset of $G_S$ in $G$. Since $|S| = p$ is odd and $S = -S$, $S$ contains an involution, say, $x$, and we can write $S = G_S + x$. Since $\Cay(G,S)$ is connected, we have $G = \langle S \rangle = G_S \cup (G_S + x) = S \oplus \{0, x\}$. Thus, $n = 2p$, $|C| = 2$ and $S  \cap G_S = \emptyset$. It is easy to see that $\Cay(G,S)$ is isomorphic to the complete bipartite graph $K_{p, p}$ with bipartition $\{G_S, G_S + x\}$ and $C$ consists of one element of $G_S$ and one element of $G_S + x$.

In the sequel we assume that $S$ is aperiodic. Then $C$ must be periodic; that is, $H = G_C$ is a nontrivial subgroup of $G$. Since $G = S \oplus C$, by Lemma \ref{lem:perodiv}, $G / H = (C / H) \oplus (S / H)$ and $H \cap (S - S) = \{0\}$. The latter implies that $|S / H| = |S| = p$. Since $G$ is abelian, $G / H$ is isomorphic to a subgroup of $G$. Since $G$ is a good abelian group, it follows that $G / H$ is a good abelian group. By Lemma \ref{lem:rmvperi}, $C / H$ is an aperiodic subset of $G/H$. Since $G/H = (C / H) \oplus (S / H)$ is a good abelian group, it follows that $S/H$ must be periodic. That is, the subgroup of periods ${\cal K}$ of $S/H$ in $G / H$ satisfies $|{\cal K}| > 1$. Since $|{\cal K}|$ divides $|S/H| = p$, we have $|{\cal K}| = p$. Since $S/H$ and ${\cal K}$ have the same cardinality, we obtain that $S/H$ is a coset of ${\cal K}$ in $G/H$. Since $|S/H| = p$ is odd and $S/H = -(S/H)$, $S/H$ contains an involution of $G/H$, say, $H+y$, and we can write $S/H = {\cal K} + (H+y)$. Since $\Cay(G,S)$ is connected, we have $G = \langle S \rangle$ and hence $G/H = \langle S/H \rangle = {\cal K} \cup ({\cal K} + (H+y))$. Therefore, $|G/H| = 2|{\cal K}| = 2p < n$ (as $H$ is nontrivial). Since $G / H = (C / H) \oplus (S/H)$, by Lemma \ref{lem:DirPprop} we have $|G/H| = |C/H| |S/H|$, which implies $|C/H| = 2p/p = 2$. Since $C$ is the union of some cosets of $H$ in $G$, it follows that $C$ is the union of exactly two cosets of $H$ in $G$. Of course, we have $|C| = 2|H| > 2$. If $H \cap S \ne \emptyset$, say, $s \in H \cap S$, then $H = H + s \in S/H = {\cal K} + (H+y)$, which contradicts the fact that $H \in {\cal K}$ and $\{{\cal K}, {\cal K} + (H+y)\}$ forms a partition of $G/H$. Thus, $H \cap S = \emptyset$. 

Assume $0 \in C$. Then $C = H \cup (H+z)$ for some $z \in G \setminus H$ and hence $C/H = \{H, H+z\}$. If $H+z \in {\cal K} + (H+y) = S/H$, then $(C / H) \oplus (S/H) = S/H \ne G/H$, which is a contradiction. So, we must have $H+z \in {\cal K}$. Since $|{\cal K}| = p$ is a prime, it follows that ${\cal K} = \la H+z \ra$ and $p$ is the smallest positive integer such that $pz \in H$. Since $C$ is a total perfect code in $\Cay(G,S)$, it induces a matching in this graph. Since $H \cap S = \emptyset$, this matching must be between $H$ and $H+z$ and therefore $|(H+z) \cap S| = 1$.
\qed
\end{proof}

In view of Theorem \ref{thm:ATPCsubgp}, it would be natural to ask whether every subset $C$ of a good abelian group $G$ such that $|C| > 2$, $|G/G_C| = 2p < n$, and $C$ is the union of exactly two cosets of $G_C$ in $G$, is a total perfect code in some connected Cayley graph of $G$. The following example gives a negative answer to this question. This shows that the counterpart of part (a) of Theorem \ref{thm:APCsubgp} for total perfect codes is not true. 
 
\begin{example}
It is known that $G = \mathbb{Z}_2 \times \mathbb{Z}_{6}$ is a good abelian group. For the subset $C = \{(0, 0), (1, 3)\} \cup \{(0, 2), (1, 5)\}$ of $G$, we have $G_C = \{(0, 0), (1,3)\}$ and $C = G_C \cup (G_C + (0,2))$. Obviously, $|C| > 2$ and $|G/G_C| = 2\cdot 3 < 12$. We claim that $C$ is not a total perfect code in any connected Cayley graph of $G$. Suppose otherwise. Then by Lemmas \ref{lem:DirP} and \ref{lem:DirPprop} such a Cayley graph $\Cay(G,S)$ should satisfy $|S| = |G|/|C| = 3$ and $(S-S) \cap (C-C) = \{(0,0)\}$. Since $(0,0) \not \in S = -S$, $S$ consists of three involutions or one involution and one pair of inverse elements. However, the differences between the only four pairs of inverse elements of $G$ are equal to $(0,1) - (0,5) = (0,4) - (0,2) = (1,1) - (1,5) = (1,4) - (1,2) = (0,2) \in C-C$. Since $(S-S) \cap (C-C) = \{(0,0)\}$, it follows that $S = \{(1,0), (1,3), (0,3)\}$ must consist of three involutions. However, this implies that $\langle S \rangle \neq G$, which contradicts the assumption that $\Cay(G,S)$ is connected.
\end{example}


\section{Cayley graphs of abelian groups}
\label{sec:abelian}

Let 
$$ 
\l_n(x) = \prod_{1 \le k < n, \text{gcd}(k, n) = 1} (x-\z_n^k)
$$
be the $n$-th cyclotomic polynomial \cite{J14}, where $n$ is a positive integer and $\z_n = e^{2\pi i/n}$ is a primitive $n$-th root of unity. It is well known that $\lambda_n(x)$ is irreducible in $\mathbb{Z}[x]$ and 
\begin{equation}
\label{eq:cyclo}
x^n - 1  = \prod_{k | n} \lambda_k(x).
\end{equation}
Define \cite{FHZ17}
$$
f_I(x) = \sum_{i \in I} \ x^i
$$
for any nonempty finite set $I$ of nonnegative integers. It was observed in \cite[Lemma 2.2]{FHZ17} that, for any inverse-closed subset $S$ of $\mathbb{Z}_{n} \setminus \{0\}$, a subset $C$ of $\mathbb{Z}_{n}$ is a perfect code in $\Cay(\mathbb{Z}_{n},S)$ if and only if there exists $q(x) \in \mathbb{Z}[x]$ such that $f_C(x)f_{S_0}(x) = (x^n - 1)q(x) + (x^{n - 1} + \cdots + x + 1)$. It was also observed in \cite[Lemma 2.4]{FHZ17} that a subset $C$ of $\mathbb{Z}_{n}$ is a total perfect code in $\Cay(\mathbb{Z}_n, S)$ if and only if there exists $q(x) \in \mathbb{Z}[x]$ such that $f_C(x)f_S(x) = (x^n - 1)q(x) + (x^{n - 1} + \cdots + x + 1)$. Using these observations and certain properties of cyclotomic polynomials, a necessary and sufficient condition for a connected circulant graph $\Cay(\mathbb{Z}_n, S)$ with degree $p-1$ or $p^l - 1$ (respectively, $p$ or $p^l$) to admit a perfect code (respectively, total perfect code) was obtained in \cite{FHZ17}, where $p$ is an odd prime and $l$ is the exponent of $p$ in $n$. In this section we extend these results to Cayley graphs of abelian groups under certain conditions using similar approaches. 

Throughout this section we assume that 
\begin{equation}
\label{eq:G}
G = \mathbb{Z}_{n_1} \times \mathbb{Z}_{n_2} \times \dots \times \mathbb{Z}_{n_d}
\end{equation}
is an abelian group with order
$$
n = n_1 n_2 \dots n_d,
$$ 
where $n_1, n_2, \ldots, n_d \ge 2$. Set
$$
S_0 = S \cup \{(0, \ldots, 0)\}
$$
for any inverse-closed subset $S$ of $G \setminus \{(0, \ldots, 0)\}$. Define 
$$
f_A(x_1, \ldots, x_d) = \sum_{(a_1, \ldots, a_d) \in A} \ x_1^{a_1} \ldots x_d^{a_d}
$$
for any subset $A$ of $G$, with the understanding that for any $(a_1, \ldots, a_d) \in A$ each coordinate $a_i \in \ZZZ_{n_i}$ is an integer between $0$ and $n_i - 1$.

\subsection{Perfect codes}

The following lemma is a generalization of \cite[Lemma 2.2]{FHZ17}.  

\begin{lemma}
\label{lem:abelPCpolyequiv} 
Let $G$ be as in \eqref{eq:G} and let $S$ be an inverse-closed subset of $G \setminus \{(0, \ldots, 0)\}$. A subset $C$ of $G$ is a perfect code in $\Cay(G, S)$ if and only if for each pair $(I, g)$ with $\emptyset \ne I \subseteq \{1, 2, \ldots, d\}$ and $g = (g_1, \ldots, g_d) \in G$ there exists a polynomial $q_{I}^{(g)}(x_1, \ldots, x_d) \in \mathbb{Z}[x_1, \ldots, x_d]$ divisible by $(\prod_{i=1}^d x_i^{g_i})(\prod_{i \in I}(x_i^{n_i} - 1))$ such that 
\begin{equation}
\label{eq:fcfs}
f_C(x_1, \ldots, x_d)f_{S_0}(x_1, \ldots, x_d) = \sum_{\emptyset \ne I \subseteq \{1, 2, \ldots, d\}} \sum_{g \in G} q_{I}^{(g)}(x_1, \ldots, x_d) + \prod_{i=1}^d \left(\sum_{j=0}^{n_i - 1} x_i^j\right).
\end{equation}
\end{lemma}

\begin{proof} 
Suppose that $C$ is a perfect code in $\Cay(G, S)$. Then $G = S_0 \oplus C$ by part (a) of Lemma \ref{lem:DirP}. So for any $(s_1, \ldots, s_d) \in S_0$ and $(c_1, \ldots, c_d) \in C$ there exist a unique element $g = (g_1, \ldots, g_d) \in G$ and $d$ integers $(r_1^{(g)}, \ldots, r_d^{(g)})$ depending on $g$ such that 
$$
(s_1, \ldots, s_d) + (c_1, \ldots, c_d) = (g_1, \ldots, g_d) + (r_1^{(g)} n_1, \ldots, r_d^{(g)} n_d).
$$ 
For any nonempty subset $I$ of $\{1, 2, \ldots, d\}$ and any $g = (g_1, \ldots, g_d) \in G$, set
$$
q_{I}^{(g)}(x_1, \ldots, x_d) = \left(\prod_{i=1}^d x_i^{g_i}\right) \left(\prod_{i \in I}(x_i^{r_i^{(g)}n_i} - 1)\right).
$$ 
Obviously, $q_{I}^{(g)}(x_1, \ldots, x_d) \in \mathbb{Z}[x_1, \ldots, x_d]$ is divisible by $(\prod_{i=1}^d x_i^{g_i})(\prod_{i \in I}(x_i^{n_i} - 1))$. We have
\begin{eqnarray*}
f_C(x_1, \ldots, x_d) f_{S_0}(x_1, \ldots, x_d)  
& = & \sum_{(c_1, \ldots, c_d) \in C, (s_1, \ldots, s_d) \in S_0}x_1^{c_1 + s_1}\ldots x_d^{c_d + s_d} \\
& = & \sum_{g = (g_1, \ldots, g_d) \in G} x_1^{g_1 + r_1^{(g)} n_1}\ldots x_d^{g_d + r_d^{(g)}n_d} \\
& = & \sum_{g = (g_1, \ldots, g_d) \in G} \left(\prod_{i=1}^d x_i^{g_i}\right) 
\left\{\prod_{i=1}^d \left((x_i^{r_i^{(g)}n_i} - 1) + 1\right)\right\} \\ 
& = & \sum_{g = (g_1, \ldots, g_d) \in G} \left(\prod_{i=1}^d x_i^{g_i}\right)
\left\{\sum_{I \subseteq \{1, 2, \ldots, d\}} \left(\prod_{i \in I}(x_i^{r_i^{(g)}n_i} - 1)\right) \right\} \\
& = & \sum_{g = (g_1, \ldots, g_d) \in G} \left(\prod_{i=1}^d x_i^{g_i}\right)
\left\{1 + \sum_{\emptyset \ne I \subseteq \{1, 2, \ldots, d\}} \left(\prod_{i \in I}(x_i^{r_i^{(g)}n_i} - 1)\right) \right\} \\
& = & \sum_{\emptyset \ne I \subseteq \{1, 2, \ldots, d\}} 
\sum_{g \in G} q_{I}^{(g)}(x_1, \ldots, x_d) + \prod_{i=1}^d \left(\sum_{j=0}^{n_i - 1} x_i^j\right).
\end{eqnarray*}

Conversely, suppose that for each pair $(I, g)$ with $\emptyset \ne I \subseteq \{1, 2, \ldots, d\}$ and $g = (g_1, \ldots, g_d) \in G$ there exists a polynomial $q_{I}^{(g)}(x_1, \ldots, x_d) \in \mathbb{Z}[x_1, \ldots, x_d]$ divisible by $(\prod_{i=1}^d x_i^{g_i})(\prod_{i \in I}(x_i^{n_i} - 1))$ such that \eqref{eq:fcfs} holds. Setting $x_i = 1$ in \eqref{eq:fcfs} for $i \in \{1, 2, \ldots, d\}$, we obtain $|S_0| |C| = f_C(1, \ldots, 1)f_{S_0}(1, \ldots, 1) = \prod_{i = 1}^d n_i = n = |G|$. We claim that $G = S_0 + C$. Suppose otherwise. Then there exists $g = (g_1, \ldots, g_d) \in G$ such that $(s_1, \ldots, s_d) + (c_1, \ldots, c_d) \ne (g_1, \ldots, g_d) + (r_1 n_1, \ldots, r_d n_d)$ for any $(s_1, \ldots, s_d) \in S_0$, $(c_1, \ldots, c_d) \in C$, and integers $(r_1, \ldots, r_d)$. Then $f_C(x_1, \ldots, x_d)f_{S_0}(x_1, \ldots, x_d)$ does not contain any term of the form $\prod_{i = 1}^d x_i^{g_i + r_i n_i}$ for $r_i \in \mathbb{Z}$. Since $\prod_{i = 1}^d \left(\sum_{j=0}^{n_i - 1} x_i^j\right)$ contains $\prod_{i = 1}^d x_i^{g_i}$, it follows that $\sum_{\emptyset \ne I \subseteq \{1, 2, \ldots, d\}} q_{I}^{(g)}(x_1, \ldots, x_d)$ contains $-\prod_{i = 1}^d x_i^{g_i}$ but no other term of the form $\prod_{i = 1}^d x_i^{g_i + r_i n_i}$, which contradicts our assumption that $(\prod_{i=1}^d x_i^{g_i})(\prod_{i \in I}(x_i^{n_i} - 1))$ divides $q_{I}^{(g)}(x_1, \ldots, x_d)$ and $I \ne \emptyset$. This contradiction shows that $G = S_0 + C$. Since $|S_0| |C| = |G|$, we then have $G = S_0 \oplus C$ by Lemma \ref{lem:DirPprop} and therefore $C$ is a perfect code in $\Cay(G, S)$ by part (a) of Lemma \ref{lem:DirP}. 
\qed
\end{proof} 

In the special case when $d = 1$, the following lemma gives \cite[Remark 1]{OPR07} (see also \cite[Lemma 2.3]{FHZ17}). 

\begin{lemma}
\label{lem:abelPCsuff}
Let $G$ be as in \eqref{eq:G} and let $S$ be an inverse-closed subset of $G \setminus \{(0, \ldots, 0)\}$ with $\la S \ra = G$. Suppose that $|S_0|$ can be factorized as $|S_0| = m_1 m_2 \cdots m_d$, where $m_i \ge 1$ is a divisor of $n_i$ for each $i \in \{1, 2, \ldots, d\}$, such that for each pair of distinct elements $(s_1, \ldots, s_d), (s_1', \ldots, s_d')$ of $S_0$ there exists at least one $j \in \{1, 2, \ldots, d\}$ with $s_j \not \equiv s'_j$ (mod $m_j$). Then $\Cay(G,S)$ admits a perfect code.
\end{lemma}

\begin{proof} 	
Let $n_i = m_i r_i$ for each $i$, and let
$$
C = (m_1 \ZZZ_{n_1}) \times \cdots \times (m_d \ZZZ_{n_d}).
$$
Then $|S_0| |C| = |S_0| r_1 r_2 \dots r_d = n_1 n_2 \dots n_d = n$. Since $(0, \ldots, 0) \in C$, we have $C \subseteq C - C$. On the other hand, for any $(c_1, \ldots, c_d), (c_1', \ldots, c_d') \in C$, $m_i$ divides $(c_i - c_i')$ mod $n_i$ as $m_i$ divides $c_i$, $c_i'$ and $n_i$. So $((c_1 - c_1')\mod n_1, \ldots, (c_d - c_d')\mod n_d) \in C$. That is, $C - C \subseteq C$, and therefore $C - C = C$. 
	
Each element of $S_0 - S_0$ is of the form $((s_1 - s_1')\mod n_1, \ldots, (s_d - s_d')\mod n_d)$, where $(s_1, \ldots, s_d), (s_1', \ldots, s_d') \in S_0$. If $s_j \not \equiv s'_j \mod m_j$ for at least one $j \in \{1, 2, \ldots, d\}$, then $((s_1 - s_1')\mod n_1, \ldots, (s_d - s_d')\mod n_d) \not \in C = C - C$. If this holds for any pair of distinct elements $(s_1, \ldots, s_d), (s_1', \ldots, s_d')$ of $S_0$, then $(C - C) \cap (S_0 - S_0) = \{(0,\ldots,0)\}$ and hence by Lemmas \ref{lem:DirP} and \ref{lem:DirPprop}, $C$ is a perfect code in $\Cay(G, S)$. 
\qed
\end{proof} 
 
The following example shows that in general the sufficient condition in Lemma \ref{lem:abelPCsuff} is not necessary. 


\begin{example}
Let $S = \{(1,1), (1,2), (3,2), (5,2), (5,3)\} \subset \mathbb{Z}_{6} \times \mathbb{Z}_{4}$. Then $\Cay(\mathbb{Z}_{6} \times \mathbb{Z}_{4}, S)$ is connected and admits $C = \{(0,0), (0,2), (3,1), (3,3)\}$ as a perfect code. There are only two ways to factorize $|S_0| = 6$ into $m_1 m_2$ such that $m_1$ divides $6$ and $m_2$ divides $4$, namely $(m_1, m_2) = (6, 1)$ or $(3, 2)$. However, in the former case we have $((5 - 5)\mod 6, (2-3)\mod 1) = (0, 0)$, and in the latter case we have $((3 - 0)\mod 3, (2-0)\mod 2) = (0, 0)$.
\end{example}

The following theorem shows that the sufficient condition in Lemma \ref{lem:abelPCsuff} is also necessary in the case when $\Cay(G, S)$ has degree $p - 1$ for a prime $p$ and only one factor $n_i$ is divisible by $p$. This result gives \cite[Theorem 1.1]{FHZ17} in the special case when $d = 1$.  

\begin{theorem}
\label{thm:APCforP}
Let $G$ be as in \eqref{eq:G} and let $S$ be an inverse-closed subset of $G \setminus \{(0, \ldots, 0)\}$ with $\la S \ra = G$. Suppose that $|S_0| = p$ is a prime and exactly one of $n_1, n_2, \ldots, n_d$, say, $n_{t}$, is divisible by $p$. Then $\Cay(G, S)$ admits a perfect code if and only if $s_{t} \not\equiv s'_{t}$ (mod $p$) for each pair of distinct elements $(s_1, \ldots, s_d), (s_1', \ldots, s_d')$ of $S_0$.
\end{theorem}

\begin{proof}
Without loss of generality we may assume $t = 1$. Suppose that $s_{1} \not\equiv s'_{1}$ (mod $p$) for each pair of distinct elements $(s_1, \ldots, s_d), (s_1', \ldots, s_d')$ of $S_0$. Setting $m_1 = p$ and $m_i = 1$ for $i \neq 1$ in Lemma \ref{lem:abelPCsuff}, we obtain that $\Cay(G, S)$ admits a perfect code. This proves the sufficiency.
	
Now suppose that $\Cay(G, S)$ admits a perfect code $C$. By our assumption, we may write $n = ap^l$ and $n_1 = a_1p^{l}$, where $l \ge 1$ and $a_1$ is not divisible by $p$. Then $a = a_1 n_2 \cdots n_d$. By Lemma \ref{lem:abelPCpolyequiv}, for each pair $(I, g)$, where $\emptyset \ne I \subseteq \{1, 2, \ldots, d\}$ and $g = (g_1, \ldots, g_d) \in G$, there exists a polynomial $q_{I}^{(g)}(x_1, \ldots, x_d) \in \mathbb{Z}[x_1, \ldots, x_d]$ divisible by $(\prod_{i=1}^d x_i^{g_i})(\prod_{i \in I}(x_i^{n_i} - 1))$ such that 
\begin{equation}
\label{eq:cso}
f_C(x_1, \ldots, x_d)f_{S_0}(x_1, \ldots, x_d) = \sum_{\emptyset \ne I \subseteq \{1, 2, \ldots, d\}} \sum_{g \in G} q_{I}^{(g)}(x_1, \ldots, x_d) + \prod_{i=1}^d \left(\sum_{j=0}^{n_i - 1} x_i^j\right).
\end{equation}
Setting $x_i = 1$ for each $i$ in this equation, we obtain $|S_0| |C| = \prod_{i = 1}^d n_i = n$. Since $|S_0| = p$, we have $|C| = n/p = ap^{l-1}$. 

\medskip
\textsf{Claim 1.} There exists at least one $j \in \{1, 2, \ldots, l\}$ such that $\lambda_{p^j}(x_1)$ divides $f_{S_0}(x_1, 1, \ldots,1)$.

Since $p^j$ divides $n_1$ for each $j \in \{1, 2, \ldots, l\}$, by \eqref{eq:cyclo}, $\lambda_{p^j}(x_1)$ divides both $x_1^{n_1} - 1$ and $1 + x_1 + \ldots + x_1^{n_1-1} = (x_1^{n_1} - 1)/(x_1-1)$. On the other hand, by \eqref{eq:cso},  
$$
f_C(x_1,1 \ldots, 1)f_{S_0}(x_1, 1,\ldots, 1) = \sum_{g \in G} q_{\{1\}}^{(g)}(x_1, 1, \ldots, 1) + (1 + x_1 + \ldots + x_1^{n_1-1})(n_2 \dots n_d)
$$ 
and $q_{\{1\}}^{(g)}(x_1, 1, \ldots, 1)$ is divisible by $x_1^{g_1}(x_1^{n_1} - 1)$ for each $g$. It follows that $\lambda_{p^j}(x_1)$ divides $f_C(x_1,1 \ldots, 1)f_{S_0}(x_1, 1,\ldots, 1)$. Since $\lambda_{p^j}(x_1)$ is irreducible, it divides either $f_C(x_1,1 \ldots, 1)$ or $f_{S_0}(x_1, 1,\ldots, 1)$ for each $j \in \{1, 2, \ldots, l\}$. If none of these polynomials $\lambda_{p^j}(x_1)$ divides $f_{S_0}(x_1, 1, \ldots, 1)$, then $\prod_{j=1}^{l} \lambda_{p^{j}}(x_1)$ divides $f_C(x_1, 1, \ldots, 1)$ and so $\prod_{j=1}^{l} \lambda_{p^{j}}(1) = p^l$ divides $f_C(1, 1, \ldots, 1) = |C| = ap^{l-1}$, which is a contradiction. Thus, Claim 1 is proved. 
	
\medskip
\textsf{Claim 2.} For all $j \in \{2, 3, \ldots, l\}$, $\lambda_{p^{j}}(x_1)$ does not divide $f_{S_0}(x_1, 1, \ldots, 1)$.
	
Suppose otherwise. Say, $f_{S_0}(x_1, 1, \ldots, 1) = \lambda_{p^{j}}(x_1) h(x_1)$ for some $j \in \{2, 3, \ldots, l\}$ and $h(x_1) \in \ZZZ[x_1]$. Since $|S_0| = p$, we may assume
$$
S_0 = \left\{(s_1^{(1)}, \ldots, s_d^{(1)}), \ldots, (s_1^{(p)}, \ldots, s_d^{(p)})\right\}. 
$$ 
Set $r_i \equiv s_1^{(i)}$ (mod $p$) and $s_1^{(i)} = r_i + k_i p$ for each $i$, where $0 \leq r_i < p$ and $k_i \in \mathbb{Z}$. Then 
\begin{eqnarray}
f_{S_0}(x_1, 1, \ldots, 1) & = & \sum_{i=1}^{p} x_1^{s_1^{(i)}} \nonumber \\
& = & \sum_{i=1}^{p} x_1^{r_i + k_i p} \nonumber \\
& = & \sum_{i=1}^{p} \left\{(x_1^p - 1)\left(\sum_{j = 0}^{k_i - 1} x_1^{r_i + jp}\right) + x_1^{r_i}\right\} \nonumber \\
& = & \sum_{i=1}^{p} x_1^{r_i} + (x_1^p - 1) \sum_{i=1}^{p} \sum_{j = 0}^{k_i - 1} x_1^{r_i + jp}. \label{eq:fso11}
\end{eqnarray}
Since $\lambda_{p^{j}}(x_1) = \sum_{i=0}^{p-1} (x_1^{p^{j-1}})^i \equiv p$ (mod $(x_1^p - 1)$), by \eqref{eq:fso11} and the assumption $f_{S_0}(x_1, 1, \ldots, 1) = \lambda_{p^{j}}(x_1) h(x_1)$ we obtain that $\sum_{i=1}^{p} x_1^{r_i} \equiv p \hat{h}(x_1)$ (mod $(x_1^p - 1)$), where $\hat{h}(x_1)$ is the unique polynomial of degree less than $p$ such that $\hat{h}(x_1) \equiv h(x_1)$ (mod $(x_1^p - 1)$). Thus, $p$ divides $\sum_{i=1}^{p} x_1^{r_i}$, which implies that all integers $r_i$ are equal. On the other hand, since $(0, \ldots, 0) \in S_0$, at least one of these integers is $0$ and therefore $r_i = 0$ for $1 \leq i \leq p$. So $s_1^{(i)} = k_i p$ is a multiple of $p$ for $1 \leq i \leq p$, which implies that $\Cay(G, S)$ is disconnected, a contradiction. Thus, Claim 2 is proved. 

Claims 1 and 2 together imply that $\lambda_{p^j}(x_1)$ divides $f_{S_0}(x_1, 1, \ldots, 1)$ if and only if $j = 1$. On the other hand, by \eqref{eq:fso11}, we have
$$
f_{S_0}(x_1, 1, \ldots, 1) = (x_1^p - 1)g(x_1) + \sum_{i=1}^{p} x_1^{r_i}
= (x_1 - 1) \lambda_{p}(x_1)g(x_1) + \sum_{i=1}^{p} x_1^{r_i},
$$
where $g(x_1) = \sum_{i=1}^{p} \sum_{j = 0}^{k_i - 1} x_1^{r_i + jp}$.
Hence $\sum_{i=1}^{p} x_1^{r_i} \equiv 0$ (mod $\lambda_{p}(x_1)$). Since the degree of $\sum_{i=1}^{p} x_1^{r_i}$ is no more than the degree of $\lambda_{p}(x_1)$, it follows that $\sum_{j=1}^{p} x_1^{r_j} = c \lambda_{p}(x_1)$ for some integer $c$. Setting $x_1 = 1$ in this equation, we obtain $c = 1$ and hence 
$\sum_{j=1}^{p} x_1^{r_j} = \lambda_{p}(x_1) = 1 + x_1 + \ldots + x_1^{p-1}$. In other words, $\{r_1, r_2, \ldots, r_{p}\} = \{0, 1, \ldots, p-1\}$, which means that $s_{1}^{(i)} \not\equiv s_{1}^{(j)}$ (mod $p$) for distinct elements $(s_1^{(i)}, \ldots, s_d^{(i)}), (s_1^{(j)}, \ldots, s_d^{(j)}) \in S_0$, as required to establish the necessity.
\qed
\end{proof}

The condition that only one $n_i$ is divisible by $p$ cannot be removed from Theorem \ref{thm:APCforP} for otherwise the result may not be true. We illustrate this by the following example.

\begin{example}
\label{counterAPCforP} 
Let $p \geq 5$ be a prime. Then 
\begin{eqnarray*}
S = & \ \{(1, 1), (p-1, p-1), (1, p-1), (p-1, 1)\} \ \cup \\ 
& \ \{(i,0): i \in \mathbb{Z}_p, i \neq 0, (p-3)/2, (p-1)/2, (p+1)/2, (p+3)/2\}
\end{eqnarray*}
is an inverse-closed generating subset of $\mathbb{Z}_p \times \mathbb{Z}_p$ with $|S|=p-1$. So $\Cay(\mathbb{Z}_p \times \mathbb{Z}_p, S)$ is connected with degree $p-1$. Let $C = \langle ((p-1)/2, 1) \rangle$. Then $|S_0| |C| = p^2$ and $C - C = C$ as $C$ is a subgroup of $\mathbb{Z}_p \times \mathbb{Z}_p$. We have
\begin{eqnarray*}
S_0 - S_0 = & \ \{(0, 2), (2, 2), (p-2, 2), (0, p-2), (2, p-2), (p-2, p-2)\}\ \cup \\ 
& \ \{(i,0): i \in \mathbb{Z}_p\} \ \cup \ \{(i,j): (i,j) \in G, i \neq (p-1)/2, (p+1)/2, j = 1, p-1\}. 
\end{eqnarray*}
Thus, for any $(x_1, x_2) \in S_0 - S_0$, we have $x_2 \in \{0, 1, 2, p-2, p-1\}$. The only elements of $C$ with the second coordinate in $\{0, 1, 2, p-2, p-1\}$ are $(0, 0), ((p-1)/2, 1), (p-1, 2), (1, p-2)$ and $((p+1)/2, p-1)$. Therefore, $(S_0 - S_0) \cap (C - C) = (S_0 - S_0) \cap C = \{(0,0)\}$. Thus, by Lemma \ref{lem:DirPprop}, $\mathbb{Z}_p \times \mathbb{Z}_p = S_0 \oplus C$, and by Lemma \ref{lem:DirP}, $C$ is a perfect code in $\Cay(\mathbb{Z}_p \times \mathbb{Z}_p,S)$. However, $S_0$ contains $(s_1^{(1)}, s_2^{(1)}) = (1,1), (s_1^{(2)}, s_2^{(2)}) = (1, p-1)$ and $(s_1^{(3)}, s_2^{(3)}) = (p-1, 1)$, but $s_1^{(1)} \equiv s_1^{(2)}$ (mod $p$) and $s_2^{(1)} \equiv s_2^{(3)}$ (mod $p$). 

Figure \ref{fig:5.1} shows the special case when $p = 5$, for which we have $S = \{(1,1), (1,4), (4,1), (4,4)\}$ and $\Cay(\mathbb{Z}_5  \times  \mathbb{Z}_5, S)$ admits $C = \langle (2,1) \rangle = \{(0,0), (2,1), (4,2), (1,3), (3,4)\}$ as a perfect code. 
\qed
\end{example}

\begin{figure}[ht]
\centering
\vspace{-2.8cm}
\includegraphics*[height=15.0cm]{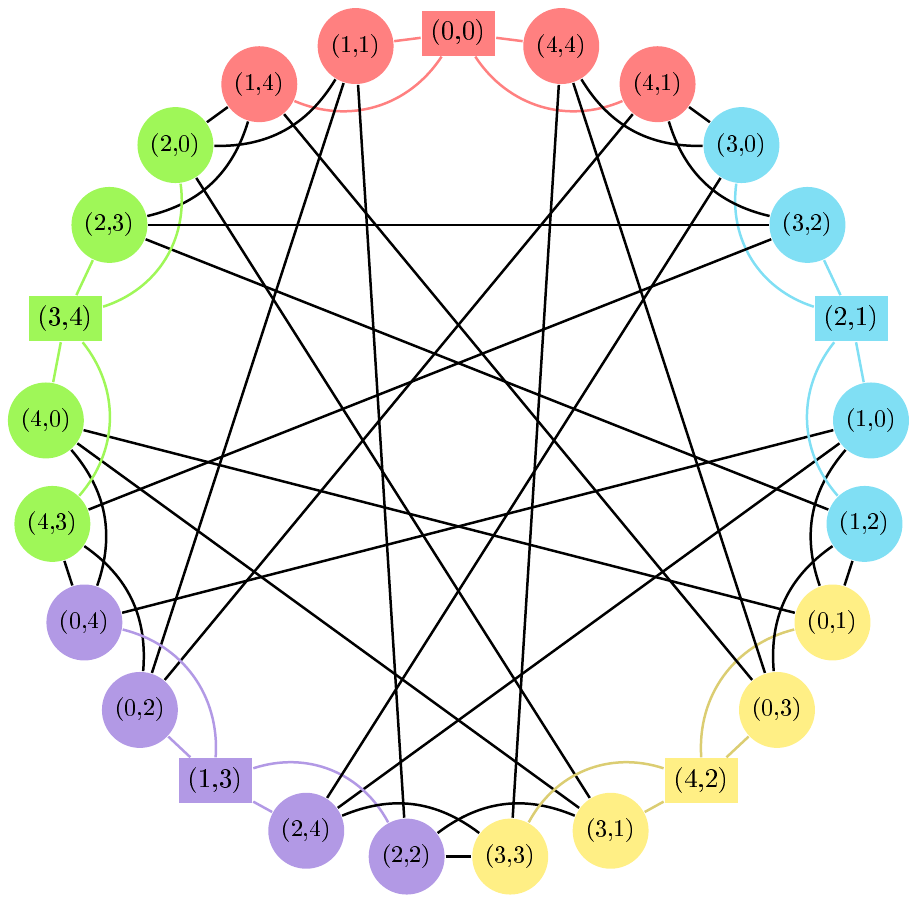}
\vspace{-2.5cm}
\caption{$\Cay(\mathbb{Z}_5 \times \mathbb{Z}_5, S)$ admits $C = \langle (2,1) \rangle$ as a perfect code, where $S = \{(1,1), (1,4), (4,1), (4,4)\}$.}
\label{fig:5.1}
\end{figure}  

The following theorem gives \cite[Theorem 1.2]{FHZ17} in the special case when $d = 1$.

\begin{theorem}
\label{thm:APCforPl}
Let $G$ be as in \eqref{eq:G} and let $S$ be an inverse-closed subset of $G \setminus \{(0, \ldots, 0)\}$ with $\la S \ra = G$. Suppose that $|S_0| = p^l$ is a prime power such that $p^l \mid n$ but $p^{l+1}\nmid n$. Suppose further that exactly one of $n_1, n_2, \ldots, n_d$, say, $n_t$, is divisible by $p$. Then $\Cay(G,S)$ admits a perfect code if and only if $s_t \not\equiv s'_t$ (mod $p^l$) for each pair of distinct elements $(s_1, \ldots, s_d), (s_1', \ldots, s_d')$ of $S_0$.
\end{theorem}

\begin{proof}	
Without loss of generality we may assume $t=1$. If $s_1 \not\equiv s'_1$ (mod $p^l$) for each pair of distinct elements $(s_1, \ldots, s_d), (s_1', \ldots, s_d')$ of $S_0$, then by setting $m_1 = p^l$ and $m_i = 1$ for each $i \neq 1$ in Lemma \ref{lem:abelPCsuff} we obtain that $\Cay(G,S)$ admits a perfect code. This proves the sufficiency. 
	
Now suppose that $\Cay(G,S)$ admits a perfect code $C$. Since $t = 1$ and $l$ is the exponent of $p$ in $n$, we have $n = ap^l$ and $n_1 = a_1p^l$, where $a_1$ is not divisible by $p$. Then $a = a_1 n_2 \cdots n_d$. By Lemma \ref{lem:abelPCpolyequiv}, for each pair $(I, g)$, where $\emptyset \ne I \subseteq \{1, 2, \ldots, d\}$ and $g = (g_1, \ldots, g_d) \in G$, there exists a polynomial $q_{I}^{(g)}(x_1, \ldots, x_d) \in \mathbb{Z}[x_1, \ldots, x_d]$ divisible by $(\prod_{i=1}^d x_i^{g_i})(\prod_{i \in I}(x_i^{n_i} - 1))$ such that 
\begin{equation}
\label{eq:csol}
f_C(x_1, \ldots, x_d)f_{S_0}(x_1, \ldots, x_d) = \sum_{\emptyset \ne I \subseteq \{1, 2, \ldots, d\}} \sum_{g \in G} q_{I}^{(g)}(x_1, \ldots, x_d) + \prod_{i=1}^d \left(\sum_{j=0}^{n_i - 1} x_i^j\right).
\end{equation}
Setting $x_i = 1$ for all $i$ in this equation, we obtain $|S_0| |C| = \prod_{i = 1}^d n_i = n$. Since $|S_0| = p^l$ and $p^{l+1}\not \mid n$, we have $|C| = n/p^l = a$ and $p$ is not a divisor of $|C|$. 

Similarly to the argument in the paragraph below Claim 1 in the proof of Theorem \ref{thm:APCforP}, by using \eqref{eq:cyclo} and \eqref{eq:csol} one can prove that $\lambda_{p^j}(x_1)$ divides either $f_C(x_1, 1, \ldots, 1)$ or $f_{S_0}(x_1, 1, \ldots, 1)$ for each $j \in \{1, 2, \ldots, l\}$. Suppose that there exists some $j \in \{1, 2, \ldots, l\}$ such that $\lambda_{p^j}(x_1)$ divides $f_C(x_1, 1, \ldots, 1)$. Then $f_C(x_1, 1, \ldots, 1) = \lambda_{p^j}(x_1)g(x_1)$ for some $g(x_1) \in \ZZZ[x_1]$. Setting $x_1 = 1$, we obtain $|C| = p g(1)$, which contradicts the fact that $p$ does not divide $|C|$. Thus, we have proved:  
	
\medskip	
\textsf{Claim 3.} For all $j \in \{1, 2, \ldots, l\}$, $\lambda_{p^j}(x_1)$ divides $f_{S_0}(x_1, 1, \ldots, 1)$.

Write 
$$
S_0 = \left\{(s_1^{(1)}, \ldots, s_d^{(1)}), \ldots, (s_1^{(p^l)}, \ldots, s_d^{(p^l)})\right\}. 
$$ 
Set $r_i \equiv s_1^{(i)}$ (mod $p^l$) and $s_1^{(i)} = r_i + k_i p^l$ for each $i$, where $0 \leq r_i < p^l$ and $k_i \in \mathbb{Z}$. Then
\begin{eqnarray*}
f_{S_0}(x_1, 1, \ldots, 1) & = & \sum_{i=1}^{p^l} x_1^{s_1^{(i)}} \\
& = & \sum_{i=1}^{p^l} x_1^{r_i + k_i p^l} \\
& = & \sum_{i=1}^{p^l} \left\{(x_1^{p^l} - 1)\left(\sum_{j=0}^{k_i - 1} x_1^{r_i + jp^l}\right) + x_1^{r_i}\right\} \\
& = & \sum_{i=1}^{p^l} x_1^{r_i} + (x_1^{p^l} - 1) \sum_{i=1}^{p^l} \sum_{j=0}^{k_i - 1} x_1^{r_i + jp^l}. 
\end{eqnarray*}
By \eqref{eq:cyclo} and Claim 3, $\prod_{j = 1}^{l}\lambda_{p^j}(x_1)$ divides both $(x_1^{p^l} - 1)$ and $f_{S_0}(x_1, 1, \ldots, 1)$. So the equation above implies that $\prod_{j = 1}^l \lambda_{p^j}(x_1)$ divides $\sum_{j=1}^{p^l} x_1^{r_j}$. Since the former has degree $p^l - 1$ and the latter has degree at most $p^l - 1$, we have $\sum_{j=1}^{p^l} x_1^{r_j} = c \prod_{j = 1}^l \lambda_{p^j}(x_1) = c\{(x_1^{p^l} - 1)/(x_1 - 1)\} = c (1 + x_1 + \ldots + x_1^{p^l - 1})$ for some integer $c$. Setting $x_1 = 1$, we obtain $c = 1$ and thus $\sum_{j=1}^{p^l} x_1^{r_j} = 1 + x_1 + \ldots + x_1^{p^l - 1}$. Therefore, $\{r_1, r_2, \ldots, r_{p^l}\} = \{0, 1, \ldots, p^l - 1\}$, or equivalently $s_{1}^{(i)} \not\equiv s_{1}^{(j)}$ (mod $p^l$) for distinct $(s_1^{(i)}, \ldots, s_d^{(i)}), (s_1^{(j)}, \ldots, s_d^{(j)}) \in S_0$, as required to prove the necessity.
\qed
\end{proof}

Unlike Theorem \ref{thm:APCforP}, at present we are not aware of any example which shows that the condition that only one $n_i$ is divisible by $p$ cannot be removed from Theorem \ref{thm:APCforPl}.

\subsection{Total perfect codes}

The following two lemmas are generalizations of Lemmas 2.4 and 2.5 in \cite{FHZ17}, respectively. We omit their proofs as they are similar to the proofs of Lemmas \ref{lem:abelPCpolyequiv} and \ref{lem:abelPCsuff}, respectively.

\begin{lemma}
\label{lem:abelTPCpolyequiv} 
Let $G$ be as in \eqref{eq:G} and let $S$ be an inverse-closed subset of $G \setminus \{(0, \ldots, 0)\}$. A subset $C$ of $G$ is a total perfect code in $\Cay(G, S)$ if and only if for each pair $(I, g)$ with $\emptyset \ne I \subseteq \{1, 2, \ldots, d\}$ and $g = (g_1, \ldots, g_d) \in G$ there exists a polynomial $q_{I}^{(g)}(x_1, \ldots, x_d) \in \mathbb{Z}[x_1, \ldots, x_d]$ divisible by $(\prod_{i=1}^d x_i^{g_i})(\prod_{i \in I}(x_i^{n_i} - 1))$ such that 
$$
f_C(x_1, \ldots, x_d)f_{S}(x_1, \ldots, x_d) = \sum_{\emptyset \ne I \subseteq \{1, 2, \ldots, d\}} \sum_{g \in G} q_{I}^{(g)}(x_1, \ldots, x_d) + \prod_{i=1}^d \left(\sum_{j=0}^{n_i - 1} x_i^j\right).
$$
\end{lemma}

\begin{lemma}
\label{lem:abelTPCsuff}
Let $G$ be as in \eqref{eq:G} and let $S$ be an inverse-closed subset of $G \setminus \{(0, \ldots, 0)\}$ with $\la S \ra = G$. Suppose that $|S|$ can be factorized as $|S| = m_1 m_2 \cdots m_d$, where $m_i \ge 1$ is a divisor of $n_i$ for each $i \in \{1, 2, \ldots, d\}$, such that for each pair of distinct elements $(s_1, \ldots, s_d), (s_1', \ldots, s_d')$ of $S$ there exists at least one $j \in \{1, 2, \ldots, d\}$ with $s_j \not \equiv s'_j$ (mod $m_j$). Then $\Cay(G,S)$ admits a total perfect code.
\end{lemma}

The sufficient condition in Lemma \ref{lem:abelTPCsuff} is not necessary even in the case when $d = 1$. For example, $\Cay(\mathbb{Z}_{12}, S)$ with $S = \{1, 3, 5, 7, 9, 11\}$ admits $\{0, 1\}$ as a perfect code but $1 \equiv 7 $ (mod $|S|$). 


The following theorem yields \cite[Theorem 1.3]{FHZ17} in the special case when $d = 1$. 

\begin{theorem}
\label{thm:ATPCforP}
Let $G$ be as in \eqref{eq:G} and let $S$ be an inverse-closed subset of $G \setminus \{(0, \ldots, 0)\}$ with $\la S \ra = G$. Suppose that $|S| = p$ is an odd prime and exactly one of $n_1, n_2, \ldots, n_d$, say, $n_{t}$, is divisible by $p$. Then $\Cay(G, S)$ admits a total perfect code if and only if $s_{t} \not\equiv s'_{t}$ (mod $p$) for each pair of distinct elements $(s_1, \ldots, s_d), (s_1', \ldots, s_d')$ of $S$.
\end{theorem}

\begin{proof}
The sufficiency follows from Lemma \ref{lem:abelTPCsuff}, and the necessity can be proved in the same way as the proof of Theorem \ref{thm:APCforP}, with $S_0$ replaced by $S$ and with the following sentences added to the end of the proof of Claim 2: Assume without loss of generality that $t=1$. Since $|S| = p$ is odd and $S = -S$, $S$ contains an involution, say, $(s_1^{(1)}, \ldots, s_d^{(1)})$, and hence $s_1^{(1)} = 0$ or $s_1^{(1)} = n_1/2$. In either case, we have $s_1^{(1)} \equiv 0$ (mod $p$), which implies $r_1 = 0$ and hence $r_i = 0$ for $1 \le i \le p$.  
\qed
\end{proof}

The following example shows that the condition that only one $n_i$ is divisible by $p$ cannot be removed from Theorem \ref{thm:ATPCforP} for otherwise the result may not be true.
	  
\begin{example}
\label{counterATPCforP}
Let $p \geq 5$ be a prime and let 
\begin{eqnarray*}
S = & \ \{(1, p-1), (p-1, p+1), (1, p+1), (p-1, p-1)\}\ \cup \\
& \ \{(i,p): i\in \mathbb{Z}_p, i \neq (p-3)/2, (p-1)/2, (p+1)/2, (p+3)/2\}.
\end{eqnarray*}
Then $S$ is an inverse-closed generating subset of $\mathbb{Z}_{p} \times \mathbb{Z}_{2p}$ with $|S| = p$. Hence $\Cay(\mathbb{Z}_{p} \times \mathbb{Z}_{2p}, S)$ is a connected Cayley graph with degree $p$. Let $C = \langle ((p-1)/2, 1) \rangle$. Then $C$ is a subgroup of $\mathbb{Z}_{p} \times \mathbb{Z}_{2p}$ with order $|C| = 2p$. So $C - C = C$ and $|S| |C| = 2p^2$. We have 
\begin{eqnarray*}
S - S = & \ \{(0, 2), (2, 2), (p-2, 2), (0, 2p-2), (2, 2p-2), (p-2, 2p-2)\} \cup \{(i,0): i \in \mathbb{Z}_p\}\ \cup \\
& \ \{(i,j): (i,j) \in G, i \neq (p-1)/2, (p+1)/2 \text{ and } j = 1, 2p-1\}. \\
\end{eqnarray*}
So each element of $S - S$ has the second coordinate $2p-2, 2p-1, 0, 1$ or $2$, but the elements of $C$ with the second coordinate in $\{2p-2, 2p-1, 0, 1, 2\}$ are precisely $(1, 2p-2), ((p+1)/2, 2p-1), (0, 0), ((p-1)/2, 1)$ and $(p-1, 2)$. Hence $(S - S) \cap (C - C) = (S - S) \cap C = \{(0,0)\}$. Thus, by Lemma \ref{lem:DirPprop}, $\mathbb{Z}_{p} \times \mathbb{Z}_{2p} = S \oplus C$, and by Lemma \ref{lem:DirP}, $C$ is a total perfect code of $\Cay(\mathbb{Z}_{p} \times \mathbb{Z}_{2p},S)$. However, $S$ contains $(s_1^{(1)}, \ldots, s_2^{(1)}) = (1,p+1), (s_1^{(2)}, \ldots, s_2^{(2)}) = (1, p-1)$ and $(s_1^{(3)}, \ldots, s_2^{(3)}) = (p-1, p+1)$, but $s_1^{(1)} \equiv s_1^{(2)}$ (mod $p$) and $s_2^{(1)} \equiv s_2^{(3)}$ (mod $p$).  

The special case $p = 5$ is illustrated in Figure \ref{fig:52}. In this case, we see that $C = \langle (2,1) \rangle = \{(0,0), (2,1), (4,2), (1,3), (3,4), (0,5), (2,6), (4,7), (1,8), (3,9)\}$ is a total perfect code of $\Cay(\mathbb{Z}_5 \times \mathbb{Z}_{10}, S)$, where $S = \{(0,5), (1,4), (4,6), (1,6), (4,4)\}$. 
\qed
\end{example}

\begin{figure}[ht]
\centering
\vspace{-2.8cm}
\includegraphics*[height=15.0cm]{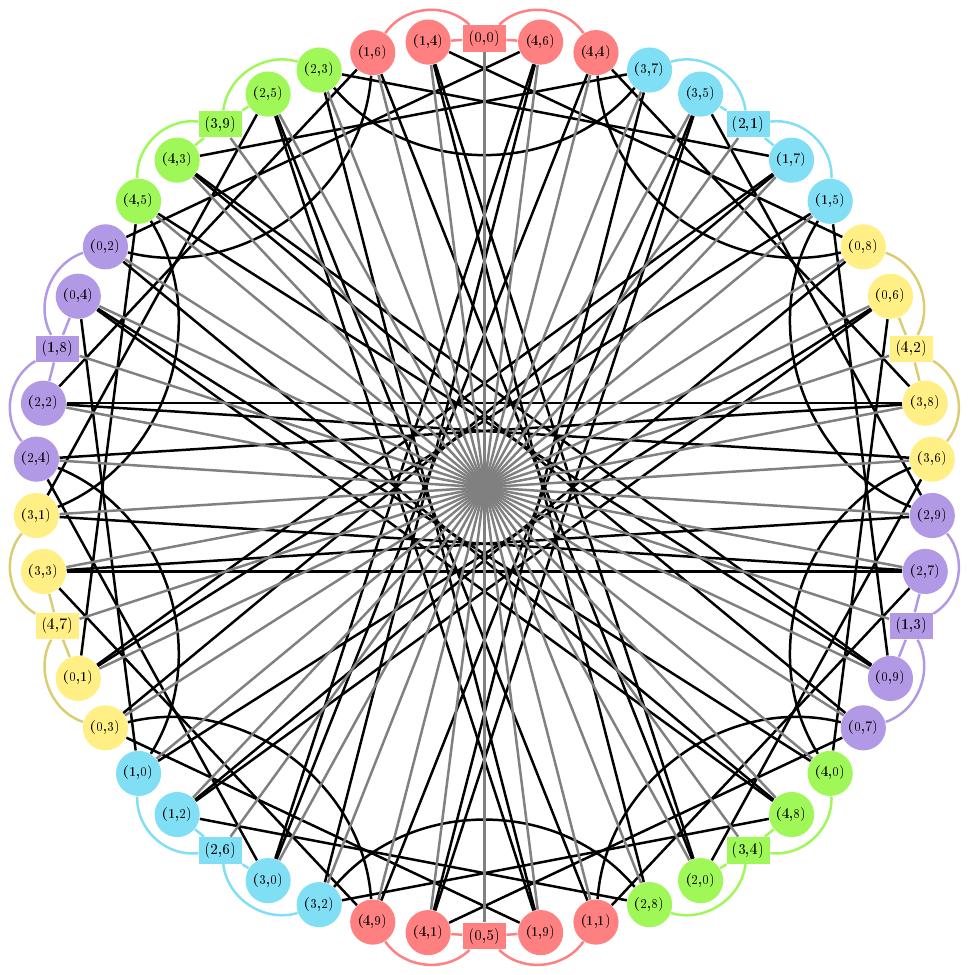}
\vspace{-2.5cm}
\caption{$\Cay(\mathbb{Z}_5 \times \mathbb{Z}_{10}, S)$ admits $C = \langle (2,1) \rangle$ as a total perfect code, where $S = \{(0,5), (1,4), (1,6), (4,4), (4,6) \}$.}
\label{fig:52}
\end{figure}  

The following theorem gives rise to \cite[Theorem 1.4]{FHZ17} in the special case when $d = 1$. 

\begin{theorem}
\label{thm:ATPCforPl}
Let $G$ be as in \eqref{eq:G} and let $S$ be an inverse-closed subset of $G \setminus \{(0, \ldots, 0)\}$ with $\la S \ra = G$. Suppose that $|S| = p^l$ is a prime power such that $p^l \mid n$ but $p^{l+1}\nmid n$. Suppose further that exactly one of $n_1, n_2, \ldots, n_d$, say, $n_t$, is divisible by $p$. Then $\Cay(G,S)$ admits a total perfect code if and only if $s_t \not\equiv s'_t$ (mod $p^l$) for each pair of distinct elements $(s_1, \ldots, s_d), (s_1', \ldots, s_d')$ of $S$.
\end{theorem}

\begin{proof}
The sufficiency follows from Lemma \ref{lem:abelTPCsuff}, and the necessity can be proved in the same way as the proof of Theorem \ref{thm:APCforPl} with $S_0$ replaced by $S$. 
\qed
\end{proof}

It is unknown whether the condition that only one $n_i$ is divisible by $p$ can be removed from Theorem \ref{thm:ATPCforPl}.

\subsection{A reduction}
\label{subsec:red}

We now prove a lemma which reduces the problem of determining whether a given connected Cayley graph $\Cay(G, S)$ of an abelian group $G$ admits a perfect code to the case when $S_0$ is aperiodic. (As proved in Lemma \ref{lem:rmvperi}, for $H = G_{S_0}$, $(S/H) \cup \{H\} \neq \emptyset$ is aperiodic in $G/H$.) This is a genuine reduction if and only if $S_0$ is periodic.

\begin{lemma}
\label{lem:redn}
Let $G$ be an abelian group. Let $S$ be a nonempty inverse-closed proper subset of $G \setminus \{0\}$ with $\la S \ra = G$, and let $H = G_{S_0}$. Then $\Cay(G, S)$ admits a perfect code if and only if $\Cay(G/H, (S/H) \setminus \{H\})$ admits a perfect code. Moreover, for any perfect code $X$ in $\Cay(G, S)$, $X/H$ is a perfect code in $\Cay(G/H, (S/H) \setminus \{H\})$; conversely, any perfect code $X/H$ in $\Cay(G/H, (S/H) \setminus \{H\})$ gives rise to $|H|^{k}$ perfect codes in $\Cay(G, S)$ each of which is formed by taking exactly one element from each coset of $H$ contained in $X/H$, where $k = |X/H|$.
\end{lemma}

\begin{proof}
Let $G, S$ and $H$ be as in the lemma. Note that $(S/H) \setminus \{H\} = (S \setminus H)/H$. Since $0 \in S_0$, we have $H \subseteq S_0$. Thus, $S \setminus H \ne \emptyset$, for otherwise we would have $H = S_0$ and so $G = \la S \ra = \la S_0 \ra = H = S_0$, which contradicts our assumption that $S$ is a proper subset of $G \setminus \{0\}$. Hence $\Cay(G/H, (S/H) \setminus \{H\})$ is well defined. 

Denote $\Ga = \Cay(G, S)$ and $\Si = \Cay(G/H, (S/H) \setminus \{H\})$. Denote by $\PP$ the set of cosets of $H$ in $G$ and $\Ga_{\PP}$ the quotient graph of $\Ga$ with respect to the partition $\PP$ of $G$. That is, $\Ga_{\PP}$ has vertex set $\PP$ such that two cosets of $H$ in $G$ are adjacent if and only if there is at least one edge of $\Ga$ between them. Since $H = G_{S_0}$, we have $S_0 + H = S_0$. Using this fact one can verify that the following hold for any pair of distinct cosets $H+x, H+y \in \PP$: $H+x$ and $H+y$ are adjacent in $\Ga_{\PP}$ $\Leftrightarrow$ there exist $a, b \in H$ such that $a+x$ and $b+y$ are adjacent in $\Ga$ $\Leftrightarrow$ there exist $a, b \in H$ and $s \in S$ such that $(a+x)-(b+y) = s$ $\Leftrightarrow$ there exist $a, b \in H$ and $s \in S$ such that $x-y = -a+s+b$ $\Leftrightarrow$ there exist $a \in H$ and $s' \in S$ such that $x-y = -a + s'$ $\Leftrightarrow$ $(H+x)-(H+y) \in (S/H) \setminus \{H\}$ $\Leftrightarrow$ $H+x$ and $H+y$ are adjacent in $\Si$. Therefore, $\Ga_{\PP}$ and $\Si$ are isomorphic. 

If two distinct cosets $H+x$ and $H+y$ are adjacent in $\Si$, then $x-y = a+s$ for some $a \in H$ and $s \in S$. Since $a + S_0 = S_0$, we have $x-y \in S$ and hence $x$ and $y$ are adjacent in $\Ga$. (Here we use the assumption that $G$ is abelian.) Since this holds for any representatives $x, y$ of $H+x, H+y$ respectively, we obtain that, for any $x' \in H+x$ and $y' \in H+y$, $x'$ and $y'$ are adjacent in $\Ga$. In other words, the bipartite subgraph of $\Ga$ between any two adjacent cosets of $H$ is a complete bipartite graph. On the other hand, the subgraph of $\Ga$ induced by the subset $H$ of $G$ is the Cayley graph $\Cay(H, S \cap H)$, and the subgraph of $\Ga$ induced by each coset $H+z$ is isomorphic to $\Cay(H, S \cap H)$. Note that $S \cap H = H \setminus \{0\}$ as $H \subseteq S_0$. Hence $\Cay(H, S \cap H)$ is a complete graph. 

Suppose that $\Ga$ admits a perfect code, say, $X$. Since $\Ga$ is connected (as $\la S \ra = G$) and the bipartite subgraph of $\Ga$ between any two adjacent cosets of $H$ is a complete bipartite graph, $X$ contains at most one vertex from each coset of $H$, that is, $|X \cap (H+z)| \le 1$ for each $z \in G$. On the other hand, as seen above the subgraph of $\Ga$ induced by each coset of $H$ is a complete graph. It follows that $X/H$ is a perfect code in $\Si$. Note that $H+x \ne H+x'$ for distinct $x, x' \in X$, for otherwise we would have $x, x' \in X \cap (H+z)$, which contradicts the fact that $|X \cap (H+z)| \le 1$. Hence $X/H$ and $X$ have the same cardinality. 

Conversely, if $\Si$ admits a perfect code, say, $X/H = \{H+x_1, \ldots, H+x_k\}$, then by taking exactly one element $y_i$ from each $H+x_i$ we obtain a prefect code $\{y_1, \ldots, y_k\}$ in $\Ga$. In this way we obtain exactly $|H|^k$ perfect codes in $\Ga$ from $X/H$. 
\qed
\end{proof} 

Lemma \ref{lem:redn} can be used to find whether a Cayley graph of an abelian group admits a perfect code by finding whether a certain quotient of it admits a perfect code, the latter being easier in some cases. As an example, let us consider $G = \ZZZ_{42}$ and 
$$S = \{1,5,6,7,11,12,13,17,18,19,23,24,25,29,30,31,35,36,37,41\} \subset G \setminus \{0\}.$$ One can see that $\Cay(G, S)$ admits $C = \{13,22\}$ as a perfect code. We have $H = G_{S_0} = 6 \ZZZ_{42}$, $G/H \cong \ZZZ_6$, $S/H = \{H, H+1, H+5\}$, and $\Cay(G/H, (S/H) \setminus \{H\}) \cong \Cay(\ZZZ_6, \{1, 5\})$ admits $C/H = \{H+1, H+4\}$ (identified with $\{1, 4\}$ in $\Cay(\ZZZ_6, \{1, 5\})$) as a perfect code.   

The ``only if'' part of the following result (\cite[Theorem 3.2]{DSLW16}) can be proved using Lemma \ref{lem:redn}: Any connected non-complete circulant $\Cay(\ZZZ_n, S)$ with degree $|S| = pq -1$ for some (not necessarily distinct) prime divisors $p, q$ of $n$ such that $n/pq$ is coprime to $pq$ and $S_0$ is periodic admits a perfect code if and only if $s \not \equiv s'$ (mod $pq$) for distinct $s, s' \in S_0$. (The ``if'' part of this result follows from the special case $d=1$ in Lemma \ref{lem:abelPCsuff}.) In fact, since $S_0$ is periodic, the subgroup of periods $H$ of $S_0$ in $G$ is nontrivial. Since $S_0$ is the union of some cosets of $H$ in $\ZZZ_n$, $|H|$ is a divisor of $|S_0| = pq$. Since $\Cay(\ZZZ_n, S)$ is connected, we have $\la S_0 \ra = \la S \ra = \ZZZ_n$. Since $\Cay(\ZZZ_n, S)$ is non-complete, $S_0$ is a proper subset of $\ZZZ_n$, which implies that $H$ is a proper subset of $S_0$. Hence $|H| \in \{p, q\}$. Without loss of generality we may assume $|H| = p$ and $S_0 = \cup_{i=0}^{q-1} (H+t_i)$, where $t_0 = 0$ and $t_i \not \equiv t_j$ (mod $(n/p)$) for distinct $i, j$. Then $\Cay(\ZZZ_n/H, (S/H) \setminus \{H\}) \cong \Cay(\ZZZ_{n/p}, T)$, where $T = \{t_0, t_1, \ldots, t_{q-1}\}$. Assume that $\Cay(\ZZZ_n, S)$ admits a perfect code. Then, by Lemma \ref{lem:redn}, $\Cay(\ZZZ_n/H, (S/H) \setminus \{H\})$ admits a perfect code. Since this circulant has degree $q-1$ and $q$ is a prime, by \cite[Theorem 3.1]{DSLW16} or \cite[Theorem 1.1]{FHZ17}, we have $t_i \not \equiv t_j$ (mod $q$) for distinct $i, j$. Since $pq$ and $n/pq$ are coprime, we then obtain $s \not \equiv s'$ (mod $pq$) for distinct $s, s' \in S_0$, establishing the necessity of the above-mentioned result. In \cite[Theorem 3.2]{DSLW16}, it was also stated (in a different way) that all perfect codes in $\Cay(\ZZZ_n, S)$ are of the form as described in Lemma \ref{lem:redn}. 

Finally, in \cite[Theorem 3.3]{YPD14} it was proved that any circulant $\Cay(\mathbb{Z}_n, S)$ such that $S_0$ is periodic and $n = |S_0| p$ for a prime $p$ admits a perfect code if and only if $s_i \not \equiv s_j$ (mod $r$) for $0 \le i < j \le r-1$, where $H$ is the subgroup of periods of $S_0$ in $\ZZZ_n$, $r = |S_0|/|H|$, and $S_0 / H = \{H+s_0, H+s_1, \ldots, H+s_{r-1}\}$ with $0 \leq s_i \leq pr - 1$ for each $i$. Moreover, under this condition all perfect codes in $\Cay(\mathbb{Z}_n, S)$ are of the form $C_{a_0, a_1, \ldots, a_{p-1}} + i$ for $\{a_0, a_1, \ldots, a_{p-1}\} \subseteq H$ and $0 \le i \le r-1$, where $C_{a_0, a_1, \ldots, a_{p-1}} = \{jr+ a_j: j = 0, 1, \ldots, p-1\}$. These results can be proved using Lemma \ref{lem:redn} and \cite[Theorem 3.2]{YPD14}.


\section{Circulant graphs}
\label{sec:tpc}

If $s \not \equiv s'$ (mod $|S_0|$) for any two distinct $s, s' \in S_0$, then $\Cay(\ZZZ_n, S)$ admits a perfect code (see \cite[Remark 1]{OPR07} and Lemma \ref{lem:abelPCsuff}), and in this case $k \ZZZ_n$ with $k = |S_0|$ is a perfect code in $\Cay(\ZZZ_n, S)$. This sufficient condition is known to be necessary as well in the following cases: (a) $|S| = p-1$, where $p$ is an odd prime (see \cite[Theorem 3.1]{DSLW16} and \cite[Theorem 1.1]{FHZ17}); (b) $|S| = p^l - 1$, where $p$ is a prime and $l$ is the exponent of $p$ in $n$ (see \cite[Theorem 3.4]{DSLW16} and \cite[Theorem 1.2]{FHZ17}); (c) $|S| = pq -1$ for some (not necessarily distinct) prime divisors $p, q$ of $n$ with $pq < n$ such that $n/pq$ and $pq$ are coprime, $\la S \ra = \ZZZ_n$, and $S_0$ is periodic (see \cite[Theorem 3.2]{DSLW16}); (d) $n \in \{p^{k}, p^{k}q, p^{2}q^{2}, pqr, p^{2}qr, pqrs: \text{ $p, q, r, s$ are primes and $k \ge 1$}\}$, $\la S \ra = \ZZZ_n$, $|S_0| < n$, and $n/|S_0|$ and $|S_0|$ are coprime (see \cite[Theorem 3.5]{DSLW16}); (e) $n/|S_0|$ is a prime and $S_0$ is aperiodic (see \cite[Theorem 3.2]{YPD14}). In this section we prove among other things that the counterparts of these results for total perfect codes are also true.   

\begin{lemma}
\label{lem:suff} 
(\cite[Lemma 2.5]{FHZ17})
Let $n \ge 4$ be an integer and $S$ an inverse-closed subset of $\ZZZ_n \setminus \{0\}$ with $\la S \ra = \ZZZ_n$. If $|S|$ divides $n$ and $s \not \equiv s'$ (mod $|S|$) for any two distinct $s, s' \in S$, then $\Cay(\mathbb{Z}_{n},S)$ admits a total perfect code. Moreover, in this case the subgroup $k \ZZZ_n$ of $\mathbb{Z}_{n}$ is a total perfect code in $\Cay(\mathbb{Z}_{n},S)$, where $k = |S|$.  
\end{lemma}

In general, the sufficient condition in this lemma is not necessary. In what follows we identify five cases for which this sufficient condition is also necessary. See Table \ref{tab2} for a summary of our results in this section along with two known results of the same flavour, where $p$ and $q$ are (not necessarily distinct) primes, $k \ge 1$ is an integer, and $|S|$ is required to be a divisor of $n$.  

{\small
\begin{table}
\begin{center}
\begin{tabular}{|c|c|c|c|c|c|c|c|}
\hline 
$n$ & $|S|$ & $C$ & $|S|$ and $|C |$ & Periodicity & All total & Reference \\
& & & coprime? & of $S$ & perfect codes &  \\
& & &  & & have been &  \\
& & &  & & determined? & \\
\hline 
 & $p$ &   & not required &   & unknown & \cite[Theorem 1.3]{FHZ17} \\ 
 & (odd) &   &   &   &  & \\ \hline
 & $p^l$ &    & required &   & unknown & \cite[Theorem 1.4]{FHZ17} \\ \hline 
 & $pq$ &    &  required & periodic & unknown & Theorem \ref{thm:CTPCforPQSP} \\ \hline
 $n \in N$ & $pq$ &    & required & aperiodic & unknown & Lemma \ref{lem:CTPCforPQSA} \\ \hline
 $n \in N$ & &   &  required &   & unknown & Theorem \ref{thm:CTPCforNbar} \\ \hline
 & & $C \le \ZZZ_n$ & not required &   & unknown & Theorem \ref{thm:CTPCforsubgp} \\ \hline
 $n = |S| p$ & & $|C|=p$ & not required & aperiodic & yes & Theorem \ref{thm:CTPCforPdA} \\ \hline	
\end{tabular}
\caption{Some cases under which $\Cay(\ZZZ_n, S)$ admits a total perfect code if and only if $s \not \equiv s'$ (mod $|S|$) for distinct $s, s' \in S$, where $N = \{p^{k}, p^{k}q, p^{2}q^{2}, pqr, p^{2}qr, pqrs: \text{ $p, q, r, s$ are primes and $k \ge 1$}\}$.}
\label{tab2}
\end{center}
\end{table}
}
  
\begin{theorem}
\label{thm:CTPCforsubgp} 
Let $n \ge 4$ be an integer and let $S$ be an inverse-closed subset of $\ZZZ_n \setminus \{0\}$ such that $\la S \ra = \ZZZ_n$ and $|S|$ divides $n$. Then $\Cay(\mathbb{Z}_n,S)$ admits a subgroup of $\mathbb{Z}_n$ as a total perfect code if and only if $s \not \equiv s'$ (mod $|S|$) for distinct $s, s' \in S$.
\end{theorem}
	
\begin{proof}
The sufficiency follows from Lemma \ref{lem:suff}. To prove the necessity, let $n$ and $S$ be as in the theorem and let $n = k|S|$. Suppose that $\Cay(\mathbb{Z}_n,S)$ admits a total perfect code $C$ which is a subgroup of $\mathbb{Z}_n$. Then $\mathbb{Z}_n = S \oplus C$ by Lemma \ref{lem:DirP}. Thus, by Lemma \ref{lem:DirPprop}, $(S-S) \cap (C-C) = \{0\}$ and $|C| = n/|S| = k$. Since the only subgroup of $\mathbb{Z}_n$ with order $k$ is $(n/k) \ZZZ_n$, we have $C = (n/k) \ZZZ_n$ and hence $C - C = (n/k) \ZZZ_n$. This together with $(S-S) \cap (C-C) = \{0\}$ implies $s-s' \not \in (n/k) \ZZZ_n$, or equivalently $s \not \equiv s'$ (mod $|S|$) for distinct $s, s' \in S$.  
\qed
\end{proof}

\begin{theorem}
\label{thm:CTPCforPdA}  
Let $n \ge 4$ be an integer. Let $S$ be an inverse-closed subset of $\ZZZ_n \setminus \{0\}$ such that $\la S \ra = \ZZZ_n$ and $n = p|S|$ for some prime $p$. Let $H$ be the subgroup of periods of $S$ in $\ZZZ_n$ under addition and set $k = |S|/|H|$. Then $\Cay(\mathbb{Z}_n, S)$ admits a total perfect code if and only if $p = 2$ and $s \not \equiv s'$ (mod $k$) for distinct $H+s, H+s' \in S/H$. Moreover, under this condition the total perfect codes in $\Cay(\mathbb{Z}_n, S)$ are precisely the subsets of $\mathbb{Z}_n$ of the form $\{a_0 + i, a_1 + k + i\}$ for $a_0, a_1 \in H$ and $0 \le i \le k-1$. 
\end{theorem}
	
\begin{proof}
We deal with the following two cases separately. 

\medskip
\textsf{Case 1.} $S$ is aperiodic. 

In this case we have $H = \{0\}$ and by Lemma \ref{lem:suff} it remains to prove the necessity.    

Suppose that $C$ is a total perfect code in $\Cay(\mathbb{Z}_n,S)$. Then $|C|$ is even and $\mathbb{Z}_n = S \oplus C$ by Lemma \ref{lem:DirP}. Thus, by Lemma \ref{lem:DirPprop}, $|C| = n/|S| = p$ and hence $p=2$. Since $\mathbb{Z}_n = S \oplus C$ and $S$ is aperiodic, by Lemma \ref{lem:peripk}, $C$ must be periodic. Denote by $K$ the subgroup of periods of $C$ in $\ZZZ_n$. Then $|K| > 1$ divides $|C| = 2$. Hence $|K| = 2$ and $K = (n/2) \ZZZ_n$. Since $|K| = |C|$, $C$ is a coset of $K$ in $\mathbb{Z}_n$. So we have proved that every total perfect code in $\Cay(\mathbb{Z}_n,S)$ is of the form $(n/2) \ZZZ_n + i = \{i, (n/2) + i\}$ for some $i \in \mathbb{Z}_n$. In particular, $\Cay(\mathbb{Z}_n,S)$ admits the subgroup $(n/2) \ZZZ_n$ of $\ZZZ_n$ as a total perfect code. Hence, by Theorem \ref{thm:CTPCforsubgp}, we have $s \not \equiv s'$ (mod $|S|$) for distinct $s, s' \in S$.  

\medskip
\textsf{Case 2.}  $S$ is periodic. 

In this case we have $|H| > 1$. Since $n/|H| = kp$, we have $H = (kp) \ZZZ_n$. We have $H \not \in S/H$, or equivalently, $S \cap H = \emptyset$, for otherwise, say, $h \in S \cap H$, we would have $-h \in S \cap H$ and hence $0 = (-h) + h \in S + h = S$, a contradiction. Let $S/H = \{H+t_1, H+t_2, \ldots, H+t_{k}\}$, where $0 < t_i \leq kp - 1$ for $1 \le i \le k$. By Lemmas \ref{lem:DirP} and \ref{lem:perodiv}, we have: $\Cay(\mathbb{Z}_n, S)$ admits a total perfect code $C$ $\Leftrightarrow$ $\mathbb{Z}_n = S \oplus C$ $\Leftrightarrow$ $\mathbb{Z}_n / H = (S / H) \oplus (C/H)$ and $H \cap (C - C) = \{0\}$ $\Leftrightarrow$ $\Cay(\mathbb{Z}_n/H, S/H)$ admits $C/H$ as a total perfect code. Since $H = (kp) \ZZZ_n$, we have $s \not \equiv s'$ (mod $k$) for distinct $H+s, H+s' \in S/H$ if and only if $t_i \not \equiv t_j$ (mod $k$) for distinct $i, j \in \{1, 2, \ldots, k\}$. 

Suppose that $\Cay(\mathbb{Z}_n, S)$ admits a total perfect code. Then $p = 2$ as seen in Case 1 and so $n/|H| = 2k$. Also, from the previous paragraph, $\Cay(\mathbb{Z}_n/H, S/H)$ admits a total perfect code. Since $\phi: \mathbb{Z}_n / H \rightarrow \mathbb{Z}_{2k}$ defined by $\phi(H+x) = x$, $0 \le x < 2k$, is a group isomorphism, it follows that $\Cay(\mathbb{Z}_{2k}, T)$ admits a total perfect code. Moreover, $T = \{t_1, t_2, \ldots, t_{k}\}$ is aperiodic as $S/H$ is aperiodic by Lemma \ref{lem:rmvperi}. It follows from what we proved in Case 1 that $t_i \not \equiv t_j$ (mod $k$) for distinct $i, j \in \{1, 2, \ldots, k\}$, as required.

Now suppose that $p = 2$ and $t_i \not \equiv t_j$ (mod $k$) for distinct $i, j \in \{1, 2, \ldots, k\}$. Then by what we proved in Case 1, $\Cay(\mathbb{Z}_{2k}, T)$, and hence $\Cay(\mathbb{Z}_n/H, S/H)$, admit a total perfect code. So $\Cay(\mathbb{Z}_n, S)$ admits a total perfect code by our discussion above. Moreover, by what we proved in Case 1, any total perfect code in $\Cay(\mathbb{Z}_{2k}, T)$ is a coset of the subgroup $k \ZZZ_{2k} = \{0, k\}$ of $\mathbb{Z}_{2k}$ and hence is of the form $k \ZZZ_{2k} + i$, $0 \le i \le k-1$. Therefore, the total perfect codes in $\Cay(\mathbb{Z}_n, S)$ are exactly the subsets of $\mathbb{Z}_n$ of the form $\{a_0 + i, a_1 + k + i\}$ for $a_0, a_1 \in H$ and $0 \le i \le k-1$.  
\qed
\end{proof}

In Theorem \ref{thm:CTPCforPdA} the condition $s \not \equiv s'$ (mod $k$) for distinct $H+s, H+s' \in S/H$ is independent of the choice of the representatives $s, s'$ of $H+s, H+s'$, respectively. In the case when $S$ is aperiodic or $S$ is periodic and $p$ is not a divisor of $|S|$, this condition is equivalent to that $s \not \equiv s'$ (mod $|S|$) for distinct $s, s' \in S$.  
 
\begin{theorem}
\label{thm:CTPCforPQSP} 
Let $n \ge 6$ be an integer, and let $p$ and $q$ be primes such that $pq$ divides $n$ and $pq$ is coprime to $n/pq$. Let $S$ be an inverse-closed subset of $\ZZZ_n \setminus \{0\}$ such that $\la S \ra = \ZZZ_n$, $|S| = pq$ and $S$ is periodic. Then $\Cay(\mathbb{Z}_n,S)$ admits a total perfect code if and only if $s \not \equiv s'$ (mod $pq$) for distinct $s, s' \in S$.
\end{theorem}
	
\begin{proof}
The sufficiency follows from Lemma \ref{lem:suff}. It remains to prove the necessity. 

Let $n = kpq$. Then $k \ge 2$ is coprime to $pq$ by our assumption. Suppose that $\Cay(\mathbb{Z}_n,S)$ admits a total perfect code $C$. Then $\mathbb{Z}_n = S \oplus C$ by Lemma \ref{lem:DirP}. Since $S$ is periodic, the subgroup of periods $H$ of $S$ in $\ZZZ_n$ has order $|H| > 1$. Since $|H|$ is a divisor of $|S| = pq$, we have $|H| \in \{pq, p, q\}$. If $|H| = pq$, then $|S| = |H| = pq$ is coprime to $n/|S|$, and thus by Lemma \ref{lem:shortpf}, $s \not \equiv s'$ (mod $pq$) for distinct $s, s' \in S$. 

It remains to consider the case when $|H| \in \{p, q\}$. Without loss of generality we may assume that $|H| = p$. Then $H = (kq) \ZZZ_n$ and $|S/H| = q$. Let $S/H = \{H+t_1, H+t_2, \ldots, H+t_q\}$ and $T = \{t_1, t_2, \ldots, t_q\}$, where $0 \leq t_i \leq kq - 1$. Note that $S \cap H = \emptyset$, for otherwise, say, $h \in S \cap H$, we would have $-h \in S \cap H$ and hence $0 = (-h) + h \in S + h = S$, a contradiction. Hence $H \not \in S/H$. Since $\mathbb{Z}_n = S \oplus C$, by Lemma \ref{lem:perodiv}, $\mathbb{Z}_n / H = (S/H) \oplus (C/H)$ and $H \cap (C-C) = \{0\}$. By Lemma \ref{lem:DirP}, $C/H$ is a total perfect code in $\Cay(\mathbb{Z}_n / H, S / H)$. Since $\phi(H+x) = x$, $H+x \in \mathbb{Z}_n / H$, defines an isomorphism $\phi$ from $\mathbb{Z}_n / H$ to $\mathbb{Z}_{kq}$, it follows that $\Cay(\mathbb{Z}_{kq}, T)$ admits a total perfect code. Since $|T| = q$ is a prime, by \cite[Theorem 1.3]{FHZ17}, we conclude that $t_i \not \equiv t_j$ (mod $q$) for distinct $t_i, t_j \in T$. 
	
For distinct $s, s' \in S$, there exist $t, t' \in T$ such that $s \in H+t$ and $s' \in H+t'$. If $t = t'$, then $s - s' = ckq$ for some integer $c$ with $0 < |c| \leq p - 1$. Since $k$ is coprime to $pq$, $p$ does not divide $ck$ and therefore $s - s' \neq 0$ (mod $pq$). If $t \neq t'$, then $s - s' = ckq + t - t'$ for some integer $c$ with $|c| \leq p - 1$. Since $t \not \equiv t'$ (mod $q$), we have $s - s' \neq 0$ (mod $q$) and thus $s - s' \neq 0$ (mod $pq$). In either case, we have  $s \not \equiv s'$ (mod $pq$) for distinct $s, s' \in S$.  
\qed
\end{proof}

\begin{lemma}
\label{lem:CTPCforPQSA}
Let $n \in \{p^{k}, p^{k}q, p^{2}q^{2}, pqr, p^{2}qr, pqrs: \text{ $p, q, r, s$ are primes and $k \ge 1$}\}$, and let $p$ and $q$ be primes such that $pq$ divides $n$ and $pq$ is coprime to $n/pq$. Let $S$ be an inverse-closed subset of $\ZZZ_n \setminus \{0\}$ such that $\la S \ra = \ZZZ_n$, $|S| = pq$ and $S$ is aperiodic. Then $\Cay(\mathbb{Z}_n,S)$ admits a total perfect code if and only if $s \not \equiv s'$ (mod $pq$) for distinct $s, s' \in S$.
\end{lemma}
	
\begin{proof}
Again, the sufficiency follows from Lemma \ref{lem:suff}. It remains to prove the necessity. 
	
Let $n = kpq$. Then $k \ge 2$ is coprime to $pq$. Suppose that $\Cay(\mathbb{Z}_n,S)$ admits a total perfect code $C$. Without loss of generality we may assume $0 \in C$. By Lemma \ref{lem:DirP}, $\mathbb{Z}_n = S \oplus C$, and by Lemma \ref{lem:DirPprop}, $|C| = n/|S| = n/pq = k$. Since $n \in N$, $\mathbb{Z}_n$ is a good cyclic group. Since $\mathbb{Z}_n = S \oplus C$ but $S$ is aperiodic, it follows that $C$ must be periodic. Denote by $H$ the subgroup of periods of $C$ in $\ZZZ_n$. Then $|H| > 1$ divides $|C| = k$. Since $0 \in C$, we have $H \subseteq C$. By Lemma \ref{lem:perodiv}, $\mathbb{Z}_n / H = (S/H) \oplus (C/H)$ and $H \cap (S-S) = \{0\}$. The latter implies $H+s \neq H+s'$ for distinct $s, s' \in S$ and thus $|S/H| = |S| = pq$. Set $\ell = n/|H|$, $S / H = \{H + t_1, H+t_2, \ldots, H+t_{pq}\}$ and $T = \{t_1, t_2, \ldots, t_{pq}\}$, where $0 \leq t_i \leq \ell - 1$ for each $i$.

We claim that $H \not \in S/H$, or equivalently, $S \cap H = \emptyset$. Suppose otherwise. Let $h \in S \cap H$. Since $-S = S$ and $H \subseteq C$, we have $h \in S \cap C$ and $-h \in S \cap C$. So $h + (-h)$ and $(-h) + h$ are two different ways to express $0$ as the sum of an element of $S$ and an element of $C$, but this contradicts the fact that $\mathbb{Z}_n = S \oplus C$. Thus $H \not \in S/H$. Since $\mathbb{Z}_n / H = (S/H) \oplus (C/H)$, by Lemma \ref{lem:DirP}, we obtain that $C/H$ is a total perfect code in $\Cay(\mathbb{Z}_n / H, S/H)$. Since $\phi: \mathbb{Z}_n / H \rightarrow \mathbb{Z}_{\ell}$ defined by $\phi(H+x) = x$ is a group isomorphism, it follows that $\Cay(\mathbb{Z}_{\ell}, T)$ admits a total perfect code. Since $\Cay(\mathbb{Z}_n, S)$ is connected, so are $\Cay(\mathbb{Z}_n / H, S/H)$ and $\Cay(\mathbb{Z}_{\ell}, T)$. Since $n \in \{p^{k}, p^{k}q, p^{2}q^{2}, pqr, p^{2}qr, pqrs: \text{ $p, q, r, s$ are primes and $k \ge 1$}\}$, we have $\ell \in \{p^{k}, p^{k}q, p^{2}q^{2}, pqr, p^{2}qr, pqrs: \text{ $p, q, r, s$ are primes and $k \ge 1$}\}$ and hence $\mathbb{Z}_n / H$ is a good group. Since $\mathbb{Z}_n / H = (S/H) \oplus (C/H)$ and $C/H$ is aperiodic by Lemma \ref{lem:rmvperi}, it follows that $S/H$ must be a periodic subset of $\mathbb{Z}_n / H$. So there exists $H+b \in \mathbb{Z}_n / H$ with $b \not \in H$ such that $(S/H) + (H+b) = S/H$. Thus, $T+b = T$, and hence $b$ is a period of $T$ in $\mathbb{Z}_{\ell}$ but is not the identity element of $\mathbb{Z}_{\ell}$ due to the isomorphism $\phi$. In other words, $T$ is periodic. Note that $\ell = (k/|H|)(pq)$ and $k/|H|$ and $|T| = pq$ are coprime. So we can apply Theorem \ref{thm:CTPCforPQSP} to $\Cay(\mathbb{Z}_{\ell}, T)$ to obtain $t_i \not \equiv t_j$ (mod $pq$) for distinct $t_i, t_j \in T$.
	
For distinct $s, s' \in S$, there exist $t, t' \in T$ such that $s \in H+t$ and $s' \in H+t'$. If $t = t'$, then $s - s' \in H$, which contradicts the fact that $H \cap (S-S) = \{0\}$. Hence $t \neq t'$ and $s - s' = c\ell + t - t'$ for some integer $c$ with $|c| \leq |H| - 1$. Since $t \not \equiv t'$ (mod $pq$) and $pq$ is a divisor of $\ell$, we obtain that $s - s' \neq 0$ (mod $pq$), as required.
\qed
\end{proof}

The following result uses Lemma \ref{lem:CTPCforPQSA} in its proof, and on the other hand it covers Lemma \ref{lem:CTPCforPQSA} as a special case. At present it is unknown whether the condition $n \in \{p^{k}, p^{k}q, p^{2}q^{2}, pqr, p^{2}qr, pqrs: \text{ $p, q, r, s$ are primes and $k \ge 1$}\}$ can be removed from this result.  

\begin{theorem}
\label{thm:CTPCforNbar} 
Let $n \in \{p^{k}, p^{k}q, p^{2}q^{2}, pqr, p^{2}qr, pqrs: \text{ $p, q, r, s$ are primes and $k \ge 1$}\}$, and let $S$ be an inverse-closed subset of $\ZZZ_n \setminus \{0\}$ such that $\la S \ra = \ZZZ_n$, $|S|$ divides $n$, and $|S|$ and $n/|S|$ are coprime. Then $\Cay(\mathbb{Z}_n,S)$ admits a total perfect code if and only if $s \not \equiv s'$ (mod $|S|$) for distinct $s, s' \in S$.
\end{theorem}

\begin{proof} 
Again, by Lemma \ref{lem:suff}, it suffices to prove the necessity. Let $n$ and $S$ be as in the theorem. Suppose that $\Cay(\mathbb{Z}_n,S)$ admits a total perfect code $C$. Then $\mathbb{Z}_n = S \oplus C$ by Lemma \ref{lem:DirP}. We aim to prove
\begin{equation}
\label{eq:ssprime}
s \not \equiv s' \text{ (mod $|S|$) for distinct $s, s' \in S$.}
\end{equation}
Since $|S|$ and $n/|S|$ are coprime, $n$ is not a prime power for otherwise we would have $|S| = 1$ and so $\Cay(\mathbb{Z}_n,S)$ is disconnected, which contradicts our assumption. Since 
$$
n \in \{p^{k}, p^{k}q, p^{2}q^{2}, pqr, p^{2}qr, pqrs: \text{ $p, q, r, s$ are primes and $k \ge 1$}\},
$$ 
one of the following cases occurs, and in each case we can easily work out possible values of $|S|$ using the assumption that $|S|$ and $n/|S|$ are coprime. 

\medskip	
\textsf{Case 1.} $n = p^kq$, where $p$ and $q$ are distinct primes and $k \geq 1$.

In this case we have $|S| \in \{p^k, q\}$ and \eqref{eq:ssprime} follows from \cite[Theorem 1.4]{FHZ17}.

\medskip	
\textsf{Case 2.} $n = p^2q^2$, where $p$ and $q$ are distinct primes.

In this case we have $|S| \in \{p^2, q^2\}$ and again \eqref{eq:ssprime} follows from \cite[Theorem 1.4]{FHZ17}.
	
\medskip
\textsf{Case 3.} $n = pqr$, where $p$, $q$ and $r$ are distinct primes.

In this case we have $|S| \in \{p, q, r, pq, pr, qr\}$ and \eqref{eq:ssprime} follows from \cite[Theorem 1.3]{FHZ17}, Theorem \ref{thm:CTPCforPQSP} or Lemma \ref{lem:CTPCforPQSA}.
	
\medskip
\textsf{Case 4.} $n = p^2qr$, where $p$, $q$  and $r$ are distinct primes.

In this case we have $|S| \in \{q, r, qr, p^2, p^2q, p^2r\}$. If $|S| \in \{q,r\}$, then we obtain \eqref{eq:ssprime} from  \cite[Theorem 1.3]{FHZ17}. If $|S| = p^2$, then \eqref{eq:ssprime} follows from \cite[Theorem 1.4]{FHZ17}. If $|S| = qr$, then \eqref{eq:ssprime} follows from Theorem \ref{thm:CTPCforPQSP} or Lemma \ref{lem:CTPCforPQSA}.
	
If $|S| \in \{p^2q, p^2r\}$ and $S$ is aperiodic, then by Lemma \ref{lem:DirPprop}, $n = |S||C| = r |S|$ or $q |S|$. So \eqref{eq:ssprime} follows from Theorem \ref{thm:CTPCforPdA}.
	
Suppose that $|S| \in \{p^2q, p^2r\}$ and $S$ is periodic. Denote by $H$ the subgroup of periods of $S$ in $\ZZZ_n$. Then $H \not \in S/H$, or equivalently, $S \cap H = \emptyset$, for otherwise, say, $h \in S \cap H$, we would have $-h \in S \cap H$ and hence $0 = (-h) + h \in S + h = S$, a contradiction. If $|H| = |S|$, then \eqref{eq:ssprime} follows from Lemma \ref{lem:shortpf}. Assume $|H| \ne |S|$ in the sequel. Then $|H|$ is a proper divisor of $|S|$ and hence $|H| \in \{p, p^2, q, pq, r, pr\}$. Set $k = |S|/|H|$ and $\ell = n/|H|$. Set $S / H = \{H+t_1, H+t_2, \ldots, H+t_{k}\}$ and $T = \{t_1, t_2, \ldots, t_{k}\}$, where $0 < t_i < \ell$. Since $\mathbb{Z}_n = S \oplus C$, by Lemma \ref{lem:perodiv}, $\mathbb{Z}_n / H = (S / H) \oplus (C/H)$ and $H \cap (C - C) = \{0\}$. Thus, by Lemma \ref{lem:DirP}, $C/H$ is a total perfect code in $\Cay(\mathbb{Z}_n / H, S / H)$. Since $\phi: \mathbb{Z}_n / H \rightarrow \mathbb{Z}_{\ell}$ defined by $\phi(H+x) = x$, $0 \le x < \ell$, is a group isomorphism, it follows that $\Cay(\mathbb{Z}_{\ell}, T)$ admits a total perfect code. Since $|T| = k \in \{p, p^2, q, r, pq, pr\}$, by \cite[Theorem 1.4]{FHZ17}, Theorem \ref{thm:CTPCforPQSP} or Lemma \ref{lem:CTPCforPQSA}, we then obtain that $t \not \equiv t'$ (mod $k$) for distinct $t, t' \in T$. For distinct $s, s' \in S$, there exist $t, t' \in T$ such that $s \in H+t$ and $s' \in H+t'$. If $t = t'$, then $s - s' = c\ell$ for some integer $c$ with $0 < |c| \leq |H| - 1$. Since $|H| > 1$ divides $|S|$ and $|S|$ is coprime to $n/|S|$, we obtain that $|S|$ is not a divisor of $c \ell$ and hence $s \not \equiv s'$ (mod $|S|$). If $t \neq t'$, then $s - s' = c\ell + t - t'$ for some integer $c$ with $|c| \leq |H| - 1$. Since $\ell = k (n/|H|)$ and $t \not \equiv t'$ (mod $k$), we have $s \not \equiv s'$ (mod $k$). Since $|S| = k |H|$, we then have $s \not \equiv s'$ (mod $|S|$).  

\medskip	
\textsf{Case 5.} $n = pqrs$, where $p$, $q$, $r$ and $s$ are distinct primes.

If $|S| \in \{p, q, r, s\}$, then we obtain \eqref{eq:ssprime} from \cite[Theorem 1.3]{FHZ17}. If $|S|$ is the product of any two of $p, q, r$ and $s$, then by Theorem \ref{thm:CTPCforPQSP} or Lemma \ref{lem:CTPCforPQSA} we obtain \eqref{eq:ssprime}. It remains to consider the case where $|S|$ is the product of three of $p, q, r$ and $s$. Without loss of generality we may assume that $|S| = qrs$. Then $|C| = n/|S|  = p$ by Lemma \ref{lem:DirPprop}. If $S$ is aperiodic, then \eqref{eq:ssprime} follows from Theorem \ref{thm:CTPCforPdA}. In the sequel we assume that $S$ is periodic; that is, the subgroup of periods $H$ of $S$ in $\ZZZ_n$ has order $|H| > 1$. Similarly to Case 4, we can see that $H \not \in S/H$. If $|H| = |S|$, then we obtain \eqref{eq:ssprime} from Lemma \ref{lem:shortpf}. Assume $|H| \ne |S|$ in the rest of the proof. Then $|H|$ is a proper divisor of $|S|$. Set $k = |S|/|H|$ and $\ell = n/|H|$. Set $S / H = \{H+t_1, H+t_2, \ldots, H+t_{k}\}$ and $T = \{t_1, t_2, \ldots, t_{k}\}$, where $0 < t_i < \ell$. Since $\mathbb{Z}_n = S \oplus C$, by Lemma \ref{lem:perodiv}, $\mathbb{Z}_n / H = (S / H) \oplus (C/H)$ and $H \cap (C - C) = \{0\}$. Hence, by Lemma \ref{lem:DirP}, $C/H$ is a total perfect code in $\Cay(\mathbb{Z}_n / H, S / H)$. Since $\phi: \mathbb{Z}_n / H \rightarrow \mathbb{Z}_{\ell}$ defined by $\phi(H+x) = x$, $0 \le x < \ell$, is a group isomorphism, it follows that $\Cay(\mathbb{Z}_{\ell}, T)$ admits a total perfect code. Since $|T| = k \in \{q, r, s, qr, qs, rs\}$, by \cite[Theorem 1.4]{FHZ17}, Theorem \ref{thm:CTPCforPQSP} or Lemma \ref{lem:CTPCforPQSA}, we have $t \not \equiv t'$ (mod $k$) for distinct $t, t' \in T$. Based on this we obtain $s \not \equiv s'$ (mod $|S|$) for distinct $s, s' \in S$ in exactly the same way as in Case 4. 
\qed
\end{proof} 	

\smallskip
\noindent \textbf{Acknowledgements}~~We are grateful to the anonymous referees for their careful reading of our manuscript and their helpful comments and suggestions. The third author was supported by a Discovery Project (DP250104965) of the Australian Research Council.

\section*{Declarations}

\noindent \textbf{Conflict of interest}~~All authors declare that they have no conflicts of interest to
disclose.


{\small

\begin{thebibliography}{99} 

\bibitem{AAK01}
R. Ahlswede, H. K. Aydinian and L. H. Khachatrian, On perfect codes and related concepts, {\em Des.  Codes Cryptogr.} {\bf 22} (2001), 221--237.

\bibitem{AD14}
C. Araujo and I. Dejter, Lattice-like total perfect codes, {\em Discuss. Math. Graph Theory} {\bf 34} (2014), no. 1, 57--74.
 
\bibitem{Biggs73}
N. Biggs, Perfect codes in graphs, {\em J. Combin. Theory Ser. B} {\bf 15} (1973), 288--296.

\bibitem{CHKS}
R. Cowen, S. H. Hechler, J. W. Kennedy and A. Steinberg, Odd neighborhood transversals on grid graphs, 
{\em Discrete Math.} {\bf 307} (2007), 2200--2208.

\bibitem{D08}
I. J. Dejter, Perfect domination in regular grid graphs, {\em Australas. J. Combin.} {\bf 42} (2008), 99--114. 

\bibitem{DS03}
I. J. Dejter and O. Serra, Efficient dominating sets in Cayley graphs, 
\textit{Discrete Appl. Math.} {\bf 129} (2003), 319--328.

\bibitem{YPD14}
Y-P. Deng, Efficient dominating sets in circulant graphs with domination number prime, {\em Inform. Process. Lett.} {\bf 114} (2014), 700--702. 

\bibitem{DSLW16}
Y-P. Deng, Y-Q. Sun, Q. Liu and H.-C. Wang, Efficient dominating sets in circulant graphs, {\em Discrete Math.} {\bf 340}(7) (2017), 1503--1507.

\bibitem{Dinitz06}
M. Dinitz, Full rank tilings of finite abelian groups, {\em SIAM J. Discrete Math.} 20 (2006), 160--170.
 
\bibitem{E87}
G. Etienne, Perfect codes and regular partitions in graphs and groups, {\em European J. Combin.} {\bf 8} (1987), 139--144.

\bibitem{Etzion11}
T. Etzion, Product constructions for perfect Lee codes, \emph{IEEE Trans. Inform. Theory} {\bf 57} (2011), 7473--7481.

\bibitem{FHZ17}
R. Feng, H. Huang and S. Zhou, Perfect codes in circulant graphs, {\em Discrete Math.} {\bf 340} (2017), 1522--1527.

\bibitem{GHT}
A-A. Ghidewon, R. H. Hammack and D. T. Taylor, Total perfect codes in tensor products of graphs, {\em Ars Combin.} {\bf 88} (2008), 129--134. 

\bibitem{Hajos1942}
G. Haj\'{o}s, \"{U}ber einfache und mehrfache Bedeckung des $n$-dimensionalen Raumes mit einem W\"{u}rfelgitter,
{\em Math. Z.} {\bf 47} (1941), 427--467.
 
\bibitem{Heden}
O.~Heden, A survey of perfect codes, {\em Adv. Math. Commun.} {\bf 2} (2008), 223--247.

\bibitem{HXZ}
H. Huang, B. Xia, S. Zhou, Perfect codes in Cayley graphs, {\em SIAM J. Discrete Math.} {\bf 32} (2018), 548--559.


\bibitem{J14}
F. Jarvis, \emph{Algebraic Number Theory}, Springer, Cham, 2014.

\bibitem{KP12}
M. Knor and P. Poto\v{c}nik, Efficient domination in cubic vertex-transitive graphs, {\em European J. Combin.} {\bf 33} (2012), no. 8, 1755--1764.

\bibitem{Kra}
J.~Kratochv\'{i}l, Perfect codes over graphs, {\em J. Combin. Theory Ser. B} {\bf 40} (1986), 224--228.

\bibitem{KPY}
D. Kuziak, I. Peterin and I. G. Yero, Efficient open domination in graph products, {\em Discrete Math. and Theoret. Comp. Sci.} {\bf 16} (2014), no. 1, 105--120.

\bibitem{KL14}
Y. S. Kwon, J. Lee, Perfect domination sets in Cayley graphs, {\em Discrete Appl. Math.} {\bf 162} (2014), 259--263.

\bibitem{L01}
J. Lee, Independent perfect domination sets in Cayley graphs, {\em J. Graph Theory} {\bf 37} (2001), 213--219.

\bibitem{KM13}
K. Reji Kumar and G. MacGillivray, Efficient domination in circulant graphs, {\em Discrete Math.} {\bf 313} (2013), 767--771. 
 
\bibitem{MBG07}
C. Mart\'{i}nez, R. Beivide and E. Gabidulin, Perfect codes for metrics induced by circulant graphs, {\em IEEE Trans. Inform. Theory} {\bf 53} (2007), 3042--3052.

\bibitem{OPR07}
N. Obradovi\'{c}, J. Peters and G. Ru\v{z}i\'{c}, Efficient domination in circulant graphs with two chord lengths, {\em Inform. Process. Lett.} {\bf 102} (2007), 253--258. 

\bibitem{RT1966}
O. Rothaus and J. G. Thompson, A combinatorial problem in the symmetric group, \textit{Pacific J. Math.}, 18 (1966), 175--178.

\bibitem{Sands62}
A. D. Sands, On the factorisation of finite abelian groups. {II},
{\em Acta Mathematica Academiae Scientiarum Hungarica} {\bf 13} (1962), no.1-2, 153--169.

\bibitem{Szab2006}
S. Szab\'{o}, Factoring finite abelian groups by subsets with maximal span, \textit{SIAM J. Discrete Math.}, 20 (2006), no. 4, 920--931.

\bibitem{SS}
S. Szab\'{o} and A. Sands, \emph{Factoring Groups into Subsets}, CRC Press, Boca Raton, FL, 2009.

\bibitem{T04}
S. Terada, Perfect codes in $\SL(2, 2^f)$, {\em European J. Combin.} {\bf 25} (2004), 1077--1085.

\bibitem{Tij95}
R. Tijdeman, Decomposition of the integers as a direct sum of two subsets, in: ``NumberTheory Seminar Paris 1992-1993'' (S. David, Ed.), pp. 261--276, Cambridge University Press, Cambridge, 1995.

\bibitem{vanLint}
J.~H.~van Lint, A survey of perfect codes, {\em Rocky Mountain J. Math.} {\bf 5} (1975), 199--224.

\bibitem{Vuza}
D. T. Vuza, Supplementary sets and regular complementary unending canons (part one), {\em Perspectives of New Music} {\bf 29} (2) (1991), 22--49. 

\bibitem{WXZ22}
Y. Wang, B. Xia and S. Zhou, Subgroup regular sets in Cayley graphs, {\em Discrete Math.} {\bf 345} (2022), no. 11, Paper No. 113023.
 
\bibitem{Z15}
S. Zhou, Cyclotomic graphs and perfect codes, {\em J. Pure Appl. Algebra} {\bf 223} (2019), 931--947.

\bibitem{Z16}
S. Zhou, Total perfect codes in Cayley graphs, {\em Des. Codes Cryptogr.} {\bf 81} (2016), 489--504.

\end{thebibliography}
}
\end{document}